\documentclass[11pt,a4paper]{article}
\usepackage{graphics}
\usepackage{amsthm}
\usepackage{enumerate}
\usepackage{fullpage}
\usepackage[colorlinks,linkcolor=blue,citecolor=blue]{hyperref}

\usepackage{amsmath}
\usepackage{amsfonts}
\usepackage[linesnumbered,lined,boxed,ruled,vlined,algo2e]{algorithm2e}
\usepackage{subfig}
\usepackage[heightadjust=all,valign=c]{floatrow}

\newtheorem{thm}{Theorem}
\newtheorem{lem}[thm]{Lemma}
\newtheorem{cor}[thm]{Corollary}

\theoremstyle{definition}
\newtheorem{rmk}[thm]{Remark}

\newcommand\trans{^{\mathit t}}
\newcommand\BII[1]{\mathit{II}_{#1}}
\newcommand\BIII[1]{\mathit{III}_{#1}}
\newcommand\MI[1]{M_{\mathit{I}}(#1)}
\newcommand\MII[1]{M_{\mathit{II}}(#1)}
\newcommand\MIII[1]{M_{\mathit{III}}(#1)}
\newcommand\MIV{M_{\mathit{IV}}}
\newcommand\MV{M_{\mathit{V}}}
\newcommand\MIast[1]{M_{\mathit{I}}^*(#1)}
\newcommand\MIIast[1]{M_{\mathit{II}}^*(#1)}
\newcommand\MIIIast[1]{M_{\mathit{III}}^*(#1)}
\newcommand\MIVast{M_{\mathit{IV}}^*}
\newcommand\MVast{M_{\mathit{V}}^*}
\newcommand\MT{M_{\mathit{T}}}
\newcommand\MTast{M_{\mathit{T}}^*}
\newcommand\ForbRow{\mathcal F_{\textup{circR}}}
\newcommand\ForbRowCol{\mathcal F_{\textup{circRC}}}
\newcommand\ARowCol{A_{\textup{circRC}}}
\newcommand\size{\mathop{\textrm{size}}}
\newcommand\miop{\mathbin{\oplus}}
\newcommand\id{\mathop{\textup{id}}}

\title{Characterization and linear-time detection of minimal obstructions to concave-round graphs and the\linebreak circular-ones property}

\author{Mart\'{\i}n D.~Safe\thanks{Departamento de Matem\'atica, Universidad Nacional del Sur, Bah\'ia Blanca, Argentina. E-mail address: \texttt{msafe@uns.edu.ar}}}

\begin{document}

\maketitle

\abstract{A graph is \emph{concave-round} if its vertices can be circularly enumerated so that the closed neighbourhood of each vertex is an interval in the enumeration. In this work, we give a minimal forbidden induced subgraph characterization for the class of concave-round graphs, solving a problem posed by Bang-Jensen, Huang, and Yeo [SIAM J Discrete Math, 13:179--193, 2000]. In addition, we show that it is possible to find one such forbidden induced subgraph in linear time in any given graph that is not concave-round. As part of the analysis, we obtain characterizations by minimal forbidden submatrices for the circular-ones property for rows and for the circular-ones property for rows and columns and show that, also for both variants of the property, one of the corresponding forbidden submatrices can be found (if present) in any given matrix in linear time. We make some final remarks regarding connections to some classes of circular-arc graphs.}

\section{Introduction}

A graph is \emph{concave-round}~\cite{MR1760336} if its vertices can be circularly enumerated $v_1,v_2,\ldots,v_n$ so that for each vertex $v_i$ there are nonnegative integers $\ell$ and $r$ (dependent on $i$) such that the closed neighborhood of $v_i$ is $\{v_{i-\ell},v_{i-\ell+1},\ldots,v_{i+r}\}$, where subindices are modulo $n$. Tucker~\cite{MR0276129,MR0309810,MR0379298} was the first to study concave-round graphs and proved that these graphs form a subclass of the class of circular-arc graphs and a superclass of the class of proper circular-arc graphs\footnote{Precise definitions of these graph classes and some basic definitions are deferred to Section~\ref{sec:defs}.}. For this reason, concave-round graphs are also known as \emph{Tucker circular-arc graphs}~\cite{MR1316450,MR1425737} or \emph{$\Gamma$ circular-arc graphs}~\cite{MR1202741,MR1776376,MR3040544}. 
There are polynomial-time algorithms for graph coloring and graph isomorphism for concave-round graphs~\cite{MR1760336,MR1776376,MR3040544}, whereas no efficient algorithms for solving these problems for circular-arc graphs are known~\cite{MR578325,MR3040544}. An algorithm for constructing canonical circular-arc models for concave-round graphs was recently given in~\cite{kobler2016solving}. A graph is \emph{convex-round}~\cite{MR1760336} if its complement is concave-round. 
Concave-round and convex-round graphs were also studied in connection to cir\-cu\-lar-per\-fect\-ness~\cite{MR1905134,MR2290743}.

Although a characterization of the class of circular-arc graphs by forbidden substructures was recently obtained in~\cite{MR3451138}, the problem of finding a characterization by forbidden induced subgraphs is still open. The problem of characterizing different subclasses of circular-arc graphs (including classes of interval graphs) by forbidden induced subgraphs has received significant attention~\cite{MR1271882,MR2536884,paperNHCA,MR3191588,MR2071482,MR2765574,MR0139159,MR2428582,MR3030589,MR0252267,MR0450140,MR0276129,MR0379298,Wegner}. In this work, we solve a problem posed by Bang-Jensen, Huang, and Yeo~\cite{MR1760336} which asks for a forbidden induced subgraph characterization for concave-round graphs (and thus also for convex-round graphs).

A binary matrix has the \emph{circular-ones property for rows}~\cite{MR0309810} (resp.~\emph{con\-sec\-u\-tive-ones property for rows}~\cite{MR0190028}) if its columns can be circularly (resp.~linearly) arranged in such a way that all the $1$'s in each row occur consecutively in the arrangement. The \emph{circular-ones property for columns} (resp.~\emph{consecutive-ones property for columns}) is defined analogously by reversing the roles of rows and columns. If no mention is made to rows or columns, we mean the corresponding property for the rows. An \emph{augmented adjacency matrix} of a graph~\cite{MR0309810} is a matrix obtained from an adjacency matrix of the graph by adding $1$'s all along the main diagonal. Clearly, a graph is concave-round if and only if its augmented adjacency matrix has the circular-ones property. The class of those graphs whose augmented adjacency matrices have the consecutive-ones property is the class of proper interval graphs~\cite{MR0252267}.

Tucker~\cite{MR0295938} gave characterizations of the consecutive-ones property for rows and for the consecutive-ones property for rows and columns by minimal forbidden submatrices. The forbidden submatrices for the consecutive-ones property for rows are known as \emph{Tucker matrices}. As already observed in~\cite{DBLP:journals/jcss/DomGN10}, no analogous characterization for the circular-ones property is available in the literature. In this work, we give such an analogous characterization. More precisely, we characterize the circular-ones property by a minimal set of minimal forbidden submatrices. Moreover, we obtain an analogous characterization also for the circular-ones property for rows and columns.

Booth and Lueker~\cite{MR0433962} devised linear-time recognition algorithms for the consecutive-ones property and the circular-ones property. McConnell~\cite{MR2290966} gave a characterization of the consecutive-ones property by the absence of odd cycles in an  associated compatibility graph and devised a linear-time algorithm for detecting one such odd cycle; these odd cycles serve as certificates of the input matrix not having the consecutive-ones property and are in a format which is especially convenient for authentication~\cite{MR3447121}. A polynomial-time algorithm for finding a Tucker submatrix (if present) in any given matrix was proposed in~\cite{TamayoMScThesis}; the first linear-time algorithm for the same task was devised in~\cite{MR3447121}. Polynomial-time algorithms for the related optimization problem of finding a Tucker submatrix of minimum size in any given matrix were proposed in~\cite{MR2930997,DBLP:journals/jcss/DomGN10}. In this work, we show, by building upon the algorithm of~\cite{MR3447121}, that one of the minimal forbidden submatrices in our characterization of the circular-ones property for rows can be found (if present) in any given matrix in linear time. Moreover, we show that an analogous result holds also for the circular-ones property for rows and columns.

As observed in~\cite{MR1760336}, the fact that the circular-ones property can be recognized in linear time implies that concave-round graphs can be recognized in linear time. The algorithm in~\cite{MR2290966} mentioned in the preceding paragraph can be used to obtain a certificate that a given graph is not concave-round~\cite{MR2554791}; such a certificate consists of an odd cycle in an associated graph. For different subclasses of circular-arc graphs for which forbidden induced subgraph characterizations are known (including classes of interval graphs), the problem of finding one of these forbidden induced subgraphs in any given graph has also been studied~\cite{paperNHCA,MR2134416,MR2765574,MR3030589,MR2399252,MR3447121,MR3318808}. For instance, in~\cite{MR3318808}, such an algorithm for the class of proper circular-arc graphs was devised.  In this work, by combining our findings about concave-round graphs and the circular-ones property together with the algorithm in~\cite{MR3318808}, we show that one of the minimal forbidden induced subgraphs for the class of concave-round graphs can be found in linear time in any graph which is not concave-round.

This work is organized as follows. In Section~\ref{sec:defs}, we give some basic definitions. In Section~\ref{sec:circ1p}, we give a characterization of the circular-ones property by a minimal set of minimal forbidden submatrices and a linear-time algorithm for finding one of these forbidden submatrices in any given matrix not having the property; we also derive analogous results for the circular-ones property for rows and columns. In Section~\ref{sec:concave-round}, we give a minimal forbidden induced subgraph characterization of concave-round graphs and show that it is possible to find one of these forbidden induced subgraphs in linear time in any given graph that is not concave-round. In Section~\ref{sec:final}, we make some final remarks regarding connections to other circular-arc graphs.

\section{Basic definitions}\label{sec:defs}

For each positive integer $k$, we denote by $[k]$ the set $\{1,\ldots,k\}$ and by $\id_k$ the identity function on $[k]$. If $S$ is a set, we denote by $\vert S\vert$ its cardinality.

\subsection*{Graphs}

All graphs in this work are simple; i.e., finite, undirected, with no loops and no multiple edges. For any basic graph-theoretic notion not defined here, the reader is referred to \cite{MR1367739}. 

Let $G$ be a graph. We denote by $V(G)$ the vertex set and by $E(G)$ the edge set of $G$. Let $v\in V(G)$. The \emph{neighborhood of $v$ in $G$}, denoted by $N_G(v)$, is the set of vertices which are adjacent to $v$ in $G$. The \emph{closed neighborhood of $v$ in $G$}, denoted by $N_G[v]$, is the set $N_G(v)\cup\{v\}$. The \emph{complement of $G$}, denoted by $\overline G$, has the same vertex set as $G$ and two different vertices are adjacent in $\overline G$ if and only if they are nonadjacent in $G$. If $X\subseteq V(G)$, the subgraph of $G$ \emph{induced by} $X$ is the graph having $X$ as vertex set and whose edges are the edges of $G$ having both endpoints in $X$. A graph class is \emph{hereditary} if it is closed under taking induced subgraphs. If $X\subseteq V(G)$, we denote by $G-X$ the graph $G[V(G)-X]$.  If $H$ is a graph, we say that $G$ \emph{contains an induced $H$} or \emph{contains $H$ as an induced subgraph} if $H$ is isomorphic to some induced subgraph of $G$. If $\mathcal H$ is a set of graphs, we say that $G$ is \emph{$\mathcal H$-free} if $G$ contains no induced $H$ for any $H\in\mathcal H$. An \emph{independent set} (resp.\ \emph{clique}) \emph{of $G$} is a set of vertices of $G$ which are pairwise nonadjacent (resp.\ adjacent). We say that $G$ is bipartite if its vertex set can be partitioned into two (possibly empty) independent sets. A \emph{co-bipartite graph} is the complement of a bipartite graph.  A \emph{walk of length $k$ in $G$} is a sequence of vertices $W=v_0,v_1,\ldots,v_k$ such that $v_{i-1}$ is adjacent to $v_i$ in $G$ for each $i\in[k]$. If so, $k$ is the \emph{length of $W$}, vertices $v_0$ and $v_k$ are the \emph{endpoints of $W$}, and $v_{i-1}$ and $v_i$, for each $i\in[k]$, are the pairs of \emph{consecutive vertices of $W$}. A walk is \emph{odd} if its length is odd, and \emph{even} otherwise. A walk is \emph{closed} it its two endpoints coincide. A \emph{path} is a walk having no repeated vertices. A \emph{cycle} is a closed walk whose only pair of repeated vertices are its endpoints. A path (resp.\ cycle) in $G$ is \emph{chordless} if there is no edge in $G$ joining two nonconsecutive vertices. The graph $P_k$ (resp.\ $C_k$) is the subgraph induced by the vertices of a chordless path (resp.\ chordless cycle) on $k$ vertices.

The \emph{intersection graph} of a family of sets $\mathcal F$ is a graph having one vertex for each member of $\mathcal F$ and having an edge joining two different vertices if and only if the corresponding members of $\mathcal F$ intersect. A \emph{circular-arc graph}~\cite{MR0309810} is the intersection graph of a set of arcs on a circle; the set of arcs is called a \emph{circular-arc model} of the graph. A circular-arc model is \emph{proper} if no two arcs of the model are one a proper subset of the other. A \emph{proper circular-arc graph}~\cite{MR0309810} is a graph admitting a proper circular-arc model.

\subsection*{Matrices and configurations}

All matrices in this work are binary matrices; i.e., having only $0$ and $1$ entries. As usual, we assume that the rows and columns of a $k\times\ell$ matrix are labeled from $1$ to $k$ and from $1$ to $\ell$, respectively. By \emph{complementing row $i$ of $M$} we mean replacing, in row $i$, all $0$ entries by $1$'s and all $1$ entries by $0$'s. The \emph{complement of $M$}, denoted $\overline M$, is the matrix arising from $M$ by replacing all $0$ entries by $1$'s and all $1$ entries by $0$'s. We denote the \emph{transpose} of a matrix $M$ by $M\trans$.

The \emph{bipartite graph associated with a matrix $M$} has one vertex for each row and one vertex for each column of $M$ and its only edges are those joining the vertex corresponding to row $i$ and the vertex corresponding to column $j$ for each $(i,j)$-entry of $M$ equal to $1$. A matrix $M$ is \emph{connected} if the bipartite graph associated with $M$ is connected and the \emph{components of $M$} are the maximal connected submatrices of $M$.

The \emph{configuration}~\cite{MR0295938} of a matrix is the set of matrices that arise from it by permutations of rows and of columns. Let $M$ and $M'$ be matrices. We say that \emph{$M$ contains $M'$ as a configuration} if some submatrix of $M$ equals $M'$ up to permutations of rows and of columns. We say that $M$ and $M'$ \emph{represent the same configuration} if $M$ and $M'$ are equal up to permutations of rows and of columns; otherwise, we say that $M$ and $M'$ \emph{represent different configurations}.

\subsection*{Algorithms}

In time and space bounds, we denote by $n$ and $m$ the number of vertices and edges, respectively, of the input graph. We say that an algorithm taking a graph as input is \emph{linear-time} if it runs in $O(n+m)$ time. If $M$ is a matrix, we denote by $\size(M)$ the sum of the number of rows, the number of columns, and the number of ones of $M$. We say that an algorithm taking a matrix $M$ as input is \emph{linear-time} if it runs in $O(\size(M))$ time. We assume that input graphs are represented by adjacency lists and input matrices are represented by lists of rows, where each row is represented by a list of the columns having a $1$ in the row. This way, graphs and matrices are represented in $O(n+m)$ and $O(\size(M))$ space, respectively.

\section{Circular-ones property and forbidden submatrices}\label{sec:circ1p}

In Subsection~\ref{ssec:circR}, we give a characterization of the circular-ones property by a minimal set of minimal forbidden submatrices and we also show that one of the minimal forbidden submatrices can be found in linear time in any given matrix not having the property. Although the number of these forbidden submatrices having $k$ rows and representing different configurations grows exponentially with $k$ (see Remark~\ref{rmk:bracelets} at the end of Subsection~\ref{ssec:circR}), we give a concise description of them; this allows us to derive, in Subsection~\ref{ssec:circRC}, a minimal set of minimal forbidden submatrices for the circular-ones property for rows and columns and show that these forbidden submatrices can also be found in linear time, if present, in any given matrix. This latter set of forbidden submatrices consists of two infinite families plus ten sporadic matrices, up to permutations of rows, of columns, and transpositions. We profit from these findings in the design of Algorithm~\ref{algo:3} in Section~\ref{sec:concave-round}.

\subsection{Forbidden submatrices for the circular-ones property}\label{ssec:circR}

Below, we state Tucker's characterization of the consecutive-ones property by minimal forbidden submatrices. The forbidden submatrices known as \emph{Tucker matrices} are displayed in Figure~\ref{fig:TuckerMatrices}, where $k$ denotes the number of rows and omitted entries are zeros. Recall the definitions regarding matrices and configurations given in Section~\ref{sec:defs}.

\begin{figure}[t!]
\ffigbox[\textwidth]{%
\begin{subfloatrow}
\subfloat[$\MI k$ for each $k\geq 3$]{%
$\MI k=\left(\begin{array}{ccccc}
             1 & 1 \\
             & 1 & 1 \\
             &   & \ddots & \ddots \\
             &   &        &     1 & 1\\
           1 & 0 & \cdots &     0 & 1
         \end{array}\right)$
}\qquad
\subfloat[t][$\MII k$ for each $k\geq 4$]{%
\begin{math}
\MII k=\left(\begin{array}{cccccc}
           1  & 1  &        &        &   & 0\\
              & 1  & 1      &        &   & 0\\
              &    & \ddots & \ddots &   & \vdots\\
              &    &        &      1 & 1 & 0\\
           1  & 1  & \cdots &      1 & 0 & 1\\
           0  & 1  & \cdots &      1 & 1 & 1
         \end{array}\right)
 \end{math}
}\end{subfloatrow}

\begin{subfloatrow}
\subfloat[$\MIII k$ for each $k\geq 3$]{%
\begin{math}
\MIII k=\left(\begin{array}{cccccc}
            1  & 1  &        &        &   & 0\\
               & 1  & 1      &        &   & 0\\
               &    & \ddots & \ddots &   & \vdots\\
               &    &        &      1 & 1 & 0\\
            0  & 1  & \cdots &      1 & 0 & 1\\
         \end{array}\right)
\end{math}
}\qquad
\subfloat[$\MIV$]{%
\begin{math}
\MIV = \left(\begin{array}{cccccc}
            1 & 1 & 0 & 0 & 0 & 0\\
            0 & 0 & 1 & 1 & 0 & 0\\
            0 & 0 & 0 & 0 & 1 & 1\\
            0 & 1 & 0 & 1 & 0 & 1
          \end{array}\right)
\end{math}
}
\end{subfloatrow}

\subfloat[$\MV$]{%
\begin{math}
\MV = \left(\begin{array}{cccccc}
             1 & 1 & 0 & 0 & 0\\
             1 & 1 & 1 & 1 & 0\\
             0 & 0 & 1 & 1 & 0\\
             1 & 0 & 0 & 1 & 1
          \end{array}\right)
\end{math}
}
}{\caption{Tucker matrices, where $k$ denotes the number of rows and omitted entries are $0$'s}\label{fig:TuckerMatrices}}
\end{figure}

\begin{thm}[\cite{MR0295938}]\label{thm:Tucker-matrices} A matrix has the consecutive-ones property if and only if it contains no Tucker matrix as a configuration.\end{thm}

Tucker~\cite{MR0309810} showed that it was possible to reduce the problem of deciding whether a matrix has the circular-ones property to that of deciding whether some matrix obtained by complementing some of its rows has the consecutive-ones property.

\begin{thm}[\cite{MR0309810}]\label{thm:cons-circ} Let $M$ be a matrix and let $M'$ be any matrix that arises from $M$ by complementing rows in such a way that some column of $M'$ consists entirely of zeros. Thus, $M$ has the circular-ones property if and only if $M'$ has the consecutive-ones property.\end{thm}

If $M$ and $M'$ satisfy the first sentence of Theorem~\ref{thm:cons-circ}, we will say that $M'$ is a \emph{Tucker reduction} of $M$. Hence, the above theorem states that a matrix $M$ has the circular-ones property if and only if any Tucker reduction of it has the consecutive-ones property. 

Booth and Lueker~\cite{MR0433962} proved that the circular-ones property can be recognized in linear time by showing that: (i) the consecutive-ones property for rows can be recognized in linear time and (ii) a Tucker reduction of a matrix can be computed also in linear time. (A different linear-time recognition algorithm for the circular-ones property not depending upon Tucker reductions was given in~\cite{MR1965521}.)

\begin{thm}[\cite{MR0433962}]\label{thm:booth-and-lueker} Both the consecutive-ones property and the circular-ones property of a matrix $M$ can be decided in $O(\size(M))$ time. A Tucker reduction $M'$ of a matrix $M$ such that $\size(M')\in O(\size(M))$ can be computed in $O(\size(M))$ time.\end{thm}

Recently, Lindzey and McConnell~\cite{MR3447121} gave a linear-time algorithm which finds a Tucker matrix in any given matrix not having the consecutive-ones property.

\begin{thm}[\cite{MR3447121}]\label{thm:Lindzey-McConnell} If a matrix $M$ does not have the consecutives-ones property, then a Tucker matrix contained in $M$ as a configuration can be found in $O(\size(M))$ time.\end{thm}

Our analysis of the circular-ones property relies on Tucker reductions. Clearly, by combining Theorems~\ref{thm:Tucker-matrices} and \ref{thm:cons-circ}, one obtains a characterization of the circular-ones property by forbidden submatrices of any Tucker reduction of it. However, our interest is on characterizing the circular-ones property by forbidden submatrices of the original matrix. In fact, the main result of this subsection (Theorem~\ref{thm:circR}) gives a minimal set of such minimal forbidden submatrices. In this way, we obtain an analogue of Theorem~\ref{thm:Tucker-matrices} for the circular-ones property; that such an analogue is not available in the literature was observed by Dom, Guo, and Niedermeier in~\cite{DBLP:journals/jcss/DomGN10}. In that work, they gave a condition in terms of a finite set of forbidden Tucker submatrices, which is sufficient for the circular-ones property to hold for every component of any matrix having a bounded number of ones per row. Nevertheless, their condition is neither necessary nor sufficient for the whole matrix to have the circular-ones property.

In order to state our characterization of the circular-ones property, we need some definitions. Let $a=a_1a_2\ldots a_k$ be a binary sequence of length $k$. We call the \emph{shift of $a$} to the sequence $a_2a_3\ldots a_ka_1$ and the \emph{reversal of $a$} to the sequence $a_ka_{k-1}\ldots a_1$. A \emph{binary bracelet}~\cite{MR1857399} is a lexicographically smallest element in an equivalence class of binary sequences under shifts and reversals. For each $k\geq 4$, let $A_k$ be the set of binary bracelets of length $k$. Let $A_3=\{000,111\}$. The elements of $A_3$ are binary bracelets but the two other binary bracelets of length $3$ ($001$ and $011$) do not belong to $A_3$. If $\phi:[k']\to[k]$ is an injective function, where $k'$ is a positive integer, we denote by $a_\phi$ the sequence $a_{\phi(1)},\ldots,a_{\phi(k')}$. For instance, the shift of $a$ is the sequence $a_\pi$, where $\pi$ is the \emph{shift permutation of $[k]$} defined as the function $\pi:[k]\to[k]$ such that $\pi(1)=2$, $\pi(2)=3$, \ldots, $\pi(k-1)=k$, and $\pi(k)=1$. Similarly, the reversal of $a$ is $a_\pi$, where $\pi$ is the \emph{reversal permutation of $[k]$} defined as the function $\pi:[k]\to[k]$ such that $\pi(1)=k$, $\pi(2)=k-1$, \ldots, and $\pi(k)=1$. Booth~\cite{MR585391} gave the first linear-time algorithm for computing the lexicographically smallest sequence that arises by repeatedly applying shifts to a given sequence. It immediately implies the following.

\begin{thm}[\cite{MR585391}]\label{thm:Booth} Given a binary sequence $a$ of length $k$, a permutation $\pi$ of $[k]$ such that $a_\pi$ is a bracelet and $\pi$ is a composition of shift and reversal permutations can be found in $O(k)$ time.\end{thm}

Let $M$ be a $k\times\ell$ matrix. If $a=a_1a_2\ldots a_k$ is a binary sequence of length $k$, we denote by $a\miop M$ the matrix that arises from $M$ by complementing those rows $i\in\{1,\ldots k\}$ such that $a_i=1$. We denote by $M^*$ the $k\times(\ell+1)$ matrix that arises from $M$ by adding one last column consisting entirely of $0$'s. 
We denote $(\MI k)^*$ by $\MIast k$; analogous conventions apply for the remaining Tucker matrices. 
We say a matrix $M$ is a \emph{minimal forbidden submatrix for the circular-ones property} if $M$ does not have the circular-ones property but every submatrix of $M$ different from $M$ has the circular-ones property. Below, we state our characterization of the circular-ones property in terms of the following set of minimal forbidden submatrices:
\[ \ForbRow=\{a\miop\MIast k:k\geq 3\mbox{ and }a\in A_k\}\cup\{\MIV,\overline{\MIV},\MVast,\overline{\MVast}\}. \]

\begin{thm}\label{thm:circR} A matrix $M$ has the circular-ones property if and only if $M$ contains no matrix in the set $\ForbRow$ as a configuration. Moreover, $\ForbRow$ is a minimal set having this property and the matrices in the set $\ForbRow$ are minimal forbidden submatrices for the circular-ones property.\end{thm}

The proof is given near the end of this subsection after some results. We need to introduce some more terminology. Let $a=a_1\ldots a_k$ and $a'=a'_1\ldots a'_k$ be two binary sequences of the same length $k$. We denote by $a+a'$ the binary sequence $(a_1+a'_1)\ldots(a_k+a'_k)$ of length $k$ where sums are taken modulo $2$. Clearly, $a\miop(a'\miop M)=(a+a')\miop M$ for each matrix $M$ having $k$ rows.

Let $M$ be a $k\times\ell$ matrix. A \emph{row map of $M$} is an injective function $\rho:[k']\to[k]$ for some positive integer $k'$. A \emph{column map of $M$} is an injective function $\sigma:[\ell']\to[\ell]$ for some positive integer $\ell'$. If $\rho:[k']\to[k]$ is a row map of $M$ and $\sigma:[\ell']\to[\ell]$ is a column map of $M$, we denote by $M_{\rho,\sigma}$ the $k'\times\ell'$ binary matrix such that, for each $(i,j)\in [k']\times[\ell']$, its $(i,j)$-entry is the $(\rho(i),\sigma(j))$-entry of $M$. Notice that $M$ contains $M'$ as a configuration if and only if there is a row map $\rho$ and column map $\sigma$ of $M$ such that $M_{\rho,\sigma}=M'$. If, in addition, $a$ is a binary sequence whose length equals the number of rows of $M$, then $(a\miop M)_{\rho,\sigma}=a_\rho\miop M_{\rho,\sigma}$. It is also clear that, if $\rho$ and $\sigma$ are a row map and a column map of $M$ and $\rho'$ and $\sigma'$ are a row map and a column map of $M_{\rho,\sigma}$, then $\rho'\circ\rho$ and $\sigma'\circ\sigma$ are a row map and a column map of $M$, and $M_{\rho'\circ\rho,\sigma'\circ\sigma}=((M_{\rho,\sigma})_{\rho',\sigma'}$. If $s$ is a positive integer and $n_1,n_2,\ldots,n_s$ are pairwise different positive integers, we denote by $\langle n_1,n_2,\ldots,n_s\rangle$ the injective function with domain $[s]$ that transforms $i$ into $n_i$ for each $i\in[s]$. If $\pi$ is a permutation of $[k]$, we denote by $\pi^\ast$ the permutation of $[k+1]$ that coincides with $\pi$ in each element of $[k]$ 
(and thus leaves $k+1$ fixed).

In the proof of the main structural result in~\cite{DBLP:journals/jcss/DomGN10}, it was observed if some of the rows of $\MVast$ are complemented, the resulting matrix contains some Tucker matrix as a configuration. For our purposes, we need a complete classification of such resulting matrices up to permutations of rows and of columns, which we will obtain in Lemma~\ref{lem:MVast}. The aforesaid proof in~\cite{DBLP:journals/jcss/DomGN10} also contains two claims which we state in the lemma below. We will profit from these connections among Tucker matrices in the proof of Lemma~\ref{lem:lemmaB}.

\begin{lem}[\cite{DBLP:journals/jcss/DomGN10}]\label{lem:DGN} For each $k\geq 4$, $00\ldots011\miop\MIIast k_{\id_k,\langle 1,\ldots,k-1,k+1,k\rangle}=\MIast k$ and, for each $k\geq 3$, $00\ldots 01\miop\MIII k=\MIast k$.\end{lem}

The two lemmas below imply that all the matrices in the set $\ForbRow$ are minimal forbidden submatrices for the circular-ones property.

\begin{lem}\label{lem:lemmaA} Let $M=\MIast k$ for some $k\geq 3$ or $M=\MVast$. If $a$ is any binary sequence of length equal to the number of rows of $M$, then $a\miop M$ is a minimal forbidden submatrix for the circular-ones property.\end{lem}
\begin{proof} Since $a\miop M$ has the circular-ones property if and only if $M$ has the circular-ones property, it suffices to prove that $M$ is a minimal forbidden submatrix for the circular-ones property. The latter follows, for instance, from Theorems~\ref{thm:Tucker-matrices} and~\ref{thm:cons-circ}.\end{proof}

\begin{lem}\label{lem:MVast} If $a$ is any binary sequence of length $4$, then $a\miop\MVast$ represents the same configuration as one of the matrices $\MIV$, $\overline{\MIV}$, $\MVast$, and $\overline{\MVast}$. Conversely, each of the matrices $\MIV$, $\overline{\MIV}$, $\MVast$, and $\overline{\MVast}$ represents the same configuration as $a\miop\MVast$ for some binary sequence $a$ of length $4$. Moreover, the four matrices $\MIV$, $\overline{\MIV}$, $\MVast$, and $\overline{\MVast}$ represent pairwise different configurations.\end{lem}
\begin{proof} For each of the $16$ possible sequences $a$, $(a\miop\MVast)$ represents the same configuration as some matrix in the set $\{\MIV,\overline{\MIV},\MVast,\overline{\MVast}\}$; namely:
\begin{gather}\label{eq:MVast}
\begin{aligned}
 \MIV  &=(0100\miop\MVast)_{\id_4,\langle 2,1,6,5,3,4\rangle}
        =(0101\miop\MVast)_{\id_4,\langle 1,2,5,6,4,3\rangle},\\
 \overline{\MIV}
       &=(1010\miop\MVast)_{\id_4,\langle 1,2,5,6,4,3\rangle}
        =(1011\miop\MVast)_{\id_4,\langle 2,1,6,5,3,4\rangle},\\
 \MVast&=(0000\miop\MVast)_{\id_4,\id_6}
        =(0001\miop\MVast)_{\id_4,\langle 2,1,4,3,6,5\rangle}\\
       &=(0110\miop\MVast)_{\langle 1,3,2,4\rangle,\langle 1,2,6,5,4,3\rangle}
        =(0111\miop\MVast)_{\langle 1,3,2,4\rangle,\langle 2,1,5,6,3,4\rangle}\\
       &=(1100\miop\MVast)_{\langle 2,1,3,4\rangle,\langle 5,6,3,4,1,2\rangle}
        =(1101\miop\MVast)_{\langle 2,1,3,4\rangle,\langle 6,5,4,3,2,1\rangle},\\
 \overline\MVast
       &=(0010\miop\MVast)_{\langle 2,1,3,4\rangle,\langle 6,5,4,3,2,1\rangle}
        =(0011\miop\MVast)_{\langle 2,1,3,4\rangle,\langle 5,6,3,4,1,2\rangle}\\
       &=(1000\miop\MVast)_{\langle 1,3,2,4\rangle,\langle 2,1,5,6,3,4\rangle}
        =(1001\miop\MVast)_{\langle 1,3,2,4\rangle,\langle 1,2,6,5,4,3\rangle}\\
       &=(1110\miop\MVast)_{\id_4,\langle 2,1,4,3,6,5\rangle}
        =(1111\miop\MVast)_{\id_4,\id_6}.
\end{aligned}
\end{gather}
In fact, the equalities for $\overline\MIV$ and $\overline\MVast$ follow from the corresponding equalities for $\MIV$ and $\MVast$ because, if we denote by $\overline a$ the sequence that arises from $a$ by replacing $0$'s by $1$'s and vice versa, then $(\overline a\miop\MVast)_{\rho,\sigma}=\overline{(a\miop\MVast)_{\rho,\sigma}}$ for any permutations $\rho$ and $\sigma$ of $[4]$ and $[6]$, respectively. Moreover, in the equalities corresponding to $\MIV$ and $\MVast$ in~\eqref{eq:MVast}, the leftmost equality in each line implies the rightmost one as follows. Let $\sigma'=\langle 2,1,4,3,6,5\rangle$ and $M\in\{\MIV,\MVast\}$. It can be verified by inspection that, in either case, $0001\miop M_{\id_4,\sigma'}=M$. Hence, if $\rho$ and $\sigma$ are maps such that $(a\miop\MVast)_{\rho,\sigma}=M$ and $\rho(4)=4$, then
$((a+0001)\miop\MVast)_{\rho,\sigma'\circ\sigma}=0001_\rho\miop(a\miop\MVast)_{\rho,\sigma'\circ\sigma}=0001\miop((a\miop\MVast)_{\rho,\sigma})_{\id_4,\sigma'}=0001\miop M_{\id_4,\sigma'}=M$. Therefore, in order to verify the validity of \eqref{eq:MVast}, it suffices to check the leftmost equality in those lines involving $a\miop\MVast$ for $a\in\{0100,0000,0110,1100\}$, which can be done by inspection.

Matrices $\MIV$, $\overline\MIV$, $\MVast$, and $\overline\MVast$ represent pairwise different configurations because they have pairwise different number of ones.\end{proof}

Our lemma below gives explicit rules for, given a Tucker matrix having $k'$ rows contained as a configuration in a Tucker reduction of a matrix $M$, finding a matrix having also $k'$ rows which is contained in $M$ as a configuration and which, by virtue of our Lemmas~\ref{lem:lemmaA} and~\ref{lem:MVast}, is a minimal forbidden submatrix for the circular-ones property.

\begin{lem}\label{lem:lemmaB} Let $M=(m_{ij})$ be a matrix, let $k$ be the number of rows of $M$, let $M'$ be a Tucker reduction of $M$, let $z$ be the label of any column of $M'$ consisting entirely of zeros, and let $a'$ be the binary sequence $m_{1z}m_{2z}\ldots m_{kz}$. Suppose that there are a row map $\rho'$ and a column map $\sigma'$ of $M'$ such that $M'_{\rho',\sigma'}$ is a Tucker matrix. If $k'$ is the number of rows of $M_{\rho',\sigma'}$, then $M_{\rho,\sigma}$ equals $a\miop\MIast{k'}$ or $a\miop\MVast$, for some binary sequence $a$ of length $k'$, $\rho=\rho'$, and some column map $\sigma$ of $M$. More precisely:
\begin{enumerate}[(i)]
\item If $M'_{\rho',\sigma'}=\MI{k'}$, then $M_{\rho,\sigma}=a\miop\MIast{k'}$, where $a=a'_\rho$, $\rho=\rho'$, and $\sigma=\langle\sigma'(1),\ldots,\sigma'(k'),z\rangle$;

\item If $M'_{\rho',\sigma'}=\MII{k'}$, then $M_{\rho,\sigma}=a\miop\MIast{k'}$, where $a=a'_\rho+00\ldots 011$, $\rho=\rho'$, and $\sigma=\langle\sigma'(1),\ldots,\sigma'(k'-1),z,\sigma'(k')\rangle$;

\item If $M'_{\rho',\sigma'}=\MIII{k'}$, then $M_{\rho,\sigma}=a\miop\MIast{k'}$, where $a=a'_\rho+00\ldots 01$, $\rho=\rho'$, and $\sigma=\langle\sigma'(1),\ldots,\sigma'(k'+1)\rangle$;

\item If $M'_{\rho',\sigma'}=\MIV$, then $M_{\rho,\sigma}=a\miop\MVast$, where $a=a'_\rho+0100$, $\rho=\rho'$, and $\sigma=\langle\sigma'(2),\sigma'(1),\sigma'(5),\sigma'(6),\sigma'(4),\sigma'(3)\rangle$;

\item If $M'_{\rho',\sigma'}=\MV$, then $M_{\rho,\sigma}=a\miop\MVast$, where $a=a'_\rho$, $\rho=\rho'$, and $\sigma=\langle\sigma'(1),\ldots,\sigma'(5),\linebreak z\rangle$.
\end{enumerate}\end{lem}
\begin{proof} By hypothesis, $M=a'\miop M'$. Thus, $M_{\rho,\sigma}=a'_\rho\miop M'_{\rho,\sigma}$ for each row map $\rho$ and each column map $\sigma$ of $M$. Let $\MT=M'_{\rho',\sigma'}$ and let $k'$ and $\ell'$ be the number of rows and columns of $\MT$. Since $\MT$ is a Tucker matrix, $\MT$ has no column consisting entirely of zeros and, consequently, $z$ does not belong to the image of $\sigma'$. Thus, it makes sense to consider the column map $\sigma''$ of $M$ given by $\sigma''=\langle\sigma'(1),\ldots,\sigma'(\ell'),z\rangle$. By construction, $M'_{\rho',\sigma''}=\MTast$. Hence, if $\rho'=\rho$, $\phi$ is any column map of $\MTast$, and $\sigma=\phi\circ\sigma''$, then
\[ M_{\rho,\sigma}=(a'\miop M')_{\rho,\sigma}
                 =a'_\rho\miop M'_{\rho,\sigma}
                 =a'_\rho\miop(M_{\rho',\phi\circ\sigma''})
                 =a'_\rho\miop(M'_{\rho',\sigma''})_{\id_{k'},\phi}
                 =a'_\rho\miop(\MTast)_{\id_{k'},\phi}. \]
The lemma follows by considering a column map $\phi$ of $\MTast$ chosen as follows.
\begin{enumerate}[(i)]
\item If $\MT=\MI{k'}$, we choose $\phi=\id_{k'+1}$ and, consequently,
\[ M_{\rho,\sigma}=a'_\rho\miop\MIast{k'}_{\id_{k'},\phi}
                  =a'_\rho\miop\MIast{k'}=a\miop\MIast{k'}, \]
where $a=a'_\rho$, $\rho=\rho'$, and $\sigma=\phi\circ\sigma''=\langle\sigma'(1),\ldots,\sigma'(k'),z\rangle$.

\item If $\MT=\MII{k'}$, we choose $\phi=\langle 1,\ldots,k'-1,k'+1,k'\rangle$ and, consequently, by Lemma~\ref{lem:DGN},
\[ M_{\rho,\sigma}=a'_\rho\miop\MIIast{k'}_{\id_{k'},\phi}
                  =a'_\rho\miop(00\ldots 011\miop\MIast{k'})
                  =a\miop\MIast{k'}, \]
where $a=a'_\rho+00\ldots 011$, $\rho=\rho'$, and $\sigma=\phi\circ\sigma''=\langle\sigma'(1),\ldots,\sigma'(k'-1),z,\sigma'(k')\rangle$.

\item If $\MT=\MIII{k'}$, we choose $\phi=\id_{k'+1}$ and, consequently, by Lemma~\ref{lem:DGN},
\[ M_{\rho,\sigma}=a'_\rho\miop\MIIIast{k'}_{\id_{k'},\phi}
                  =a'_\rho\miop(00\ldots 01\miop\MIast{k'})
                  =a\miop\MIast{k'}, \]
where $a=a'_\rho+00\ldots 01$, $\rho=\rho'$, and $\sigma=\phi\circ\sigma''=\langle\sigma'(1),\ldots,\sigma'(k'+1)\rangle$.

\item If $\MT=\MIV$, we choose $\phi=\langle 2,1,5,6,4,3\rangle$ and, consequently,
\[ M_{\rho,\sigma}=a'_\rho\miop(\MIVast)_{\id_{4},\phi}
                  =a'_\rho\miop(0100\miop\MVast)
                  =a\miop\MVast, \]
where $a=a'_\rho+0100$, $\rho=\rho'$, and $\sigma=\phi\circ\sigma''=\langle\sigma'(2),\sigma'(1),\sigma'(5),\sigma'(6),\sigma'(4),\sigma'(3)\rangle$.

\item If $\MT=\MV$, we choose $\phi=\id_{6}$ and, consequently,
\[ M_{\rho,\sigma}=a'_\rho\miop(\MVast)_{\id_{4},\phi}
                  =a'_\rho\miop\MVast=a\miop\MVast, \]
where $a=a'_\rho$, $\rho=\rho'$, and $\sigma=\phi\circ\sigma''=\langle\sigma'(1),\sigma'(2),\ldots,\sigma'(5),z\rangle$.
\end{enumerate}
This completes the proof of the lemma.\end{proof}

Our next three lemmas point at determining which of the minimal forbidden submatrices of the form $a\miop\MIast k$ in the above lemma represent the same or different configurations. Recall that if $\pi$ is a permutation of $[k]$, we denote by $\pi^\ast$ the permutation of $[k+1]$ that coincides with $\pi$ in each element of $[k]$.

\begin{lem}\label{lem:a=a'-mod-rot&perm} Let $k\geq 3$ and let $a$ and $a'$ be two binary sequences of length $k$. If $a$ equals $a'$ up to shifts and reversals, then $a\miop\MIast k$ and $a'\miop\MIast k$ represent the same configuration. More precisely, if $\pi$ is a composition of shift and reversal permutations of $[k]$, then $a_\pi\miop\MIast k=(a\miop\MIast k)_{\pi,\pi^*}$.\end{lem}
\begin{proof} It suffices to consider the case where $\pi$ is the shift or the reversal permutation of $[k]$ (as the general case will then follow by induction). In either case, it is easy to verify that $\MIast k_{\pi,\pi^*}=\MIast k$. Therefore, $(a\miop\MIast k)_{\pi,\pi^*}=a_\pi\miop\MIast k_{\pi,\pi^*}=a_\pi\miop\MIast k$.\end{proof}

\begin{lem}\label{lem:MIast3} If $a$ is a binary sequence of length $3$, then $a\miop\MIast 3$ represent the same configuration as $a'\miop\MIast 3$ for some $a'\in A_3$. Moreover, $000\miop\MIast 3$ and $111\miop\MIast 3$ represent different configurations.\end{lem}
\begin{proof} Let $a'$ be the lexicographically smallest binary sequence of length $3$ such that $a\miop\MIast 3$ and $a'\miop\MIast{3}$ represent the same configuration. By Lemma~\ref{lem:a=a'-mod-rot&perm}, $a'$ is a bracelet; i.e., $a'\in\{000,001,011,111\}$. Since $(011\miop\MIast 3)_{\id_3,\langle 2,1,4,3\rangle}=000\miop\MIast 3$ and $(001\miop\MIast 3)_{\id_3,\langle 3,4,1,2\rangle}=111\miop\MIast 3$, the first assertion follows. The second assertion follows, for instance, from the fact that among $000\miop\MIast 3$ and $111\miop\MIast 3$ only the former has a column consisting entirely of zeros.\end{proof}

\begin{lem}\label{lem:MIastgeq4} Let $k\geq 4$ and let $a$ and $a'$ be binary sequences of length $k$. If $a\miop\MIast k$ and $a'\miop\MIast k$ represent the same configuration, then $a'$ equals $a$ up to shifts and reversals.\end{lem}
\begin{proof} Let $a=a_1\ldots a_k$ and $a'=a'_1\ldots a'_k$ and suppose that $a\miop\MIast k$ and $a'\miop\MIast k$ represent the same configuration. Thus, there are permutations $\pi$ and $\sigma$ of $[k]$ and $[k+1]$, respectively, such that $(a\miop\MIast k)_{\pi,\sigma}=a'\miop\MIast k$.  Hence, 
\[ \MIast k_{\pi,\sigma}=(a\miop(a\miop\MIast k))_{\pi,\sigma}=a_\pi\miop(a\miop\MIast k)_{\pi,\sigma}=a_\pi\miop(a'\miop\MIast k)=a''\miop\MIast k, \]
where $a''=a_\pi+a'$. Since each of $\MIast k_{\pi,\sigma}$ and $\MIast k$ has at least five columns and exactly two ones per row, necessarily $a''$ consists entirely of zeros. Therefore, $a'=a_\pi$ and $\MIast k_{\pi,\sigma}=\MIast k$. Thus, for each $i\in[k]$, the entries $(\pi(i),\sigma(i))$ and $(\pi(i-1),\sigma(i))$ of $\MIast k$ are ones, where $\pi(0)$ stands for $\pi(k)$. Hence, $\pi(i)=\pi(i-1)\pm 1\pmod k$, for each $i\in[k]$. As $k\geq 4$ and $\pi$ is a permutation, either $\pi(i)=\pi(k)+i\pmod k$ for each $i\in[k]$ or $\pi(i)=\pi(k)-i\pmod k$ for each $i\in[k]$. In both cases, $\pi$ is a composition of shift and reversal permutations. Since $a'=a_\pi$, the proof of the lemma is complete.
\end{proof}

By combining the preceding lemmas with Theorems~\ref{thm:cons-circ}, \ref{thm:booth-and-lueker}, and \ref{thm:Lindzey-McConnell}, we now show that it is possible to find in linear time a matrix in the set $\ForbRow$ contained as a configuration in any matrix not having the circular-ones property.

\begin{algorithm2e}[t!]\DontPrintSemicolon
\SetAlgoVlined
 \KwIn{A $k\times\ell$ matrix $M=(m_{ij})$ not having the circular-ones property}
 \KwOut{Maps $\rho_F$ and $\sigma_F$ such that $M_{\rho_F,\sigma_F}\in\ForbRow$ and the sequence $c$ of the entries in the last column of $M_{\rho_F,\sigma_F}$}

 Let $M'$ be a Tucker reduction of $M$ such that $\size(M')\in O(\size(M))$, $z$ be the label of a column of $M'$ full of $0$'s, and $a':=m_{1z}m_{2z}\ldots m_{kz}$\;\label{algo1:line1} 
 Find maps $\rho'$ and $\sigma'$ such that $M'_{\rho',\sigma'}$ is a Tucker matrix and let $k'$ be the number of rows of $M'_{\rho',\sigma'}$\;\label{algo1:line2}
 Let $\rho:=\rho'$ and find a map $\sigma$ and a sequence $a$ such that $M_{\rho,\sigma}$ equals $a\miop\MIast{k'}$ or $a\miop\MVast$\;\label{algo1:line3}
 \If{$M_{\rho,\sigma}=a\miop\MIast{k'}$}{\label{algo1:step4cond}
   Find a composition $\pi$ of shift and reversal permutations of $[k']$ such that $a_\pi$ is a bracelet\;\label{algo1:step4}
   \If{$a_\pi=001$}{\label{algo1:step4a}
      \Return $\rho_F:=\pi\circ\rho$, $\sigma_F:=\langle 3,4,1,2\rangle\circ\pi^*\circ\sigma$, and $c:=111$\;\label{algo1:step4c}}
   \ElseIf{$a_\pi=011$}{
      \Return $\rho_F:=\pi\circ\rho$, $\sigma_F:=\langle 2,1,4,3\rangle\circ\pi^*\circ\sigma$, and $c:=000$\;\label{algo1:step4e}}
   \Else{\Return $\rho_F:=\pi\circ\rho$, $\sigma_F:=\pi^*\circ\sigma$, and $c:=a_\pi$\;\label{algo1:step4z}}}
 \Else{
   Let $\rho'$ and $\sigma'$ be maps so that $(a\miop\MVast)_{\rho',\sigma'}\in\{\MIV,\overline\MIV,\MVast,\overline\MVast\}$\;\label{algo1:step5a}
   Let $c' $ be the sequence of entries in the last column of $(a\miop\MVast)_{\rho',\sigma'}$\;\label{algo1:step5b}
   \Return $\rho_F:=\rho'\circ\rho$, $\sigma_F:=\sigma'\circ\sigma$, and $c:=c'$\;\label{algo1:step5}}
 \caption{Finds some matrix in the set $\ForbRow$ contained in $M$ as a configuration}\label{algo:1}
\end{algorithm2e}

\begin{cor}\label{cor:circR} Given any matrix $M$ not having the circular-ones property, Algorithm~\ref{algo:1} finds a matrix in the set $\ForbRow$ contained in $M$ as a configuration. Moreover, Algorithm~\ref{algo:1} can be implemented to run in $O(\size(M))$ time.\end{cor}
\begin{proof} We first prove the correctness. First suppose that the condition of line~\ref{algo1:step4cond} holds and let $\pi$ be as specified in line~\ref{algo1:step4}. Thus, by Lemma~\ref{lem:a=a'-mod-rot&perm},
\[ M_{\pi\circ\rho,\pi^*\circ\sigma}=(M_{\rho,\sigma})_{\pi,\pi^*}=(a\miop\MIast{k'})_{\pi,\pi^*}=a_\pi\miop\MIast{k'}_{\pi,\pi^*}=a_\pi\miop\MIast{k'}, \]
which, by definition, has $a_\pi$ as the sequence of entries in its last column. 
Moreover, as $a_\pi$ is a bracelet, $a_\pi\in A_{k'}$ unless $a_\pi\in\{001,011\}$.
This proves that the output given in line~\ref{algo1:step4z} is correct. If $a_\pi=001$, then $M_{\pi\circ\rho,\langle 3,4,1,2\rangle\circ\pi^*\circ\sigma}
=(M_{\pi\circ\rho,\pi^*\circ\sigma})_{\id_3,\langle 3,4,1,2\rangle}
=(a_\pi\miop\MIast{k'})_{\id_3,\langle 3,4,1,2\rangle}=(001\miop\MIast 3)_{\id_3,\langle 3,4,1,2\rangle}=111\miop\MIast 3$, which proves that the output given in line~\ref{algo1:step4c} is correct. Similarly, if $a_\pi=011$, then $M_{\pi\circ\rho,\langle 2,1,4,3\rangle\circ\pi^*\circ\sigma}=(011\miop\MIast 3)_{\id_3,\langle 2,1,4,3\rangle}=000\miop\MIast 3$, which proves that the output given in line~\ref{algo1:step4e} is also correct. It only remains to consider the case where $M_{\rho,\sigma}=a\miop\MVast$. In this case, permutations $\rho'$ and $\sigma'$ as specified in Step 5 exist by virtue of Lemma~\ref{lem:MVast} and satisfy $M_{\rho'\circ\rho,\sigma'\circ\sigma}=(M_{\rho,\sigma})_{\rho',\sigma'}=(a\miop\MVast)_{\rho',\sigma'}\in\{\MIV,\overline{\MIV},\MVast,\overline{\MVast}\}$, which proves that the output given in line~\ref{algo1:step5} is also correct. This completes the proof of the correctness.

We now prove that the algorithm can be implemented to run in linear time. Theorems~\ref{thm:booth-and-lueker} shows that line~\ref{algo1:line1} can be performed in $O(\size(M))$ time. Since $M$ does not have the circular-ones property, Theorems~\ref{thm:cons-circ} implies that $M'$ does not have the consecutive-ones property and, consequently, the algorithm of Theorem~\ref{thm:Lindzey-McConnell} can be used to find maps $\rho'$ and $\sigma'$ satisfying the condition of line~\ref{algo1:line2} in $O(\size(M))$ time. Lemma~\ref{lem:lemmaB} shows that line~\ref{algo1:line3} can be performed in additional $O(k')$ time. By virtue of Theorem~\ref{thm:Booth}, line~\ref{algo1:step4} can be performed in additional $O(k')$ time. Clearly, lines~\ref{algo1:step4a} to~\ref{algo1:step4z} can also be performed in $O(k')$ time. Finally, lines~\ref{algo1:step5a} to~\ref{algo1:step5} can be performed in $O(1)$ time (e.g., by precomputing the a correspondence $a\mapsto(\rho',\sigma',c')$ for each of the $16$ possible sequences $a$; see~\eqref{eq:MVast} in the proof of Lemma~\ref{lem:MVast}). This proves the $O(\size(M))$ time bound for Algorithm~\ref{algo:1}.\end{proof}

We are now ready to give the proof of Theorem~\ref{thm:circR}.

\begin{proof}[Proof of Theorem~\ref{thm:circR}] Lemmas~\ref{lem:lemmaA} and \ref{lem:MVast} imply that  no matrix in the set $\ForbRow$ has the circular-ones property. Thus, if $M$ contains a matrix in the set $\ForbRow$ as a configuration, then $M$ does not have the circular-ones property. Conversely, if $M$ does not have the circular-ones property, Algorithm~\ref{algo:1} can be used to find a matrix in the set $\ForbRow$ contained in $M$ as a configuration. This proves the first assertion of the theorem.

That all the matrices in the set $\ForbRow$ are minimal forbidden submatrices for the circular-ones property follows from Lemmas~\ref{lem:lemmaA} and~\ref{lem:MVast}. Hence, in order to prove that $\ForbRow$ is a minimal set of matrices satisfying the first assertion of the theorem, it suffices to prove that no two different matrices in the set $\ForbRow$ represent the same configuration, which follows from Lemmas~\ref{lem:MIast3}, \ref{lem:MIastgeq4}, and~\ref{lem:MVast}.
\end{proof}

\begin{rmk}\label{rmk:bracelets} Notice that, if $k\geq 3$, then the number of matrices in $\ForbRow$ having $k$ rows is $\vert A_k\vert$ if $k\neq 4$ and $\vert A_k\vert+4$ if $k=4$. Hence, for $k=3,4,5,6,7,8,\ldots$ this number is $2,10,8,13,18,30,\ldots$ and coincides, for every $k\geq 5$, with the number of binary bracelets of length $k$, which is known to be:
\begin{equation} \label{eq:bracelets}
   \frac{1}{2k}\sum_{d\mid k}\varphi(d)2^{k/d}+
     \left\{\begin{array}{ll}
       \frac 34\cdot 2^{k/2}          &\mbox{if $k$ is even,}\smallskip\\
       \frac 12\cdot 2^{(k-1)/2}      &\mbox{if $k$ is odd,}
     \end{array}\right.
\end{equation}
where $d\mid k$ stands for `$d$ is a positive divisor of $k$' and $\varphi$ is Euler's totient function. The interested reader is referred to~\cite{MR0098037} for a derivation of \eqref{eq:bracelets} and the definition of $\varphi$.\end{rmk}

\subsection{Forbidden submatrices for the circular-ones property for rows and columns}\label{ssec:circRC}

We say that a matrix $M$ is a \emph{minimal forbidden submatrix for the circular-ones property for rows and columns} if $M$ does not have the circular-ones property for rows and columns but every submatrix of $M$ different from $M$ has the circular-ones property for rows and columns. Let
\[ \ARowCol=\{0001,0011,0111,00001,00011,00111,01111,000111\}, \]
and let
\[ \ForbRowCol=\bigcup_{k=3}^\infty\{\MIast k,\overline{\MIast k}\}\cup\{a\miop\MIast{\vert a\vert}\colon\,a\in\ARowCol\}\cup\{\MVast,\overline{\MVast}\}, \]
where by $\vert a\vert$ we denote the length of sequence $a$. Hence, $\ForbRowCol$ consists of two infinite families of matrices ($\MIast 3,\MIast 4,\MIast 5,\ldots$ and $\overline{\MIast 3},\overline{\MIast 4},\overline{\MIast 5},\ldots$) plus ten sporadic matrices. We denote by $\ForbRowCol\trans$ the set of transposes of the matrices in the set $\ForbRowCol$. 

The result below is the main result of this subsection and characterizes the circular-ones property for rows and columns in terms of the set $\ForbRowCol\cup\ForbRowCol\trans$ of minimal forbidden submatrices.

\begin{thm}\label{thm:circRC} A matrix $M$ has the circular-ones property for rows and columns if and only if $M$ contains no matrix in the set $\ForbRowCol\cup\ForbRowCol\trans$ as a configuration. Moreover, $\ForbRowCol\cup\ForbRowCol\trans$ is a minimal set having this property and the matrices in the set $\ForbRowCol\cup\ForbRowCol\trans$ are minimal forbidden submatrices for the circular-ones property for rows and columns.\end{thm}

The proof of Theorem~\ref{thm:circRC} will be given at the end of this subsections after some results. For that purpose, we need a few more definitions. Let $\mathcal P$ be the set of sequences of $x$'s and $y$'s listed in the first column of Table~\ref{tab:patterns}. Let $a=a_1a_2\ldots a_k$ be a binary sequence of length $k$. If $b\in\mathcal P$, we say that \emph{$b$ occurs in $a$ at position $i\in[k]$} if $k\geq\vert b\vert$ and $a_i,a_{i+1},\ldots,a_{i+\vert b\vert-1}=b'$ where subindices are modulo $k$ and $b'$ is a sequence that arises from $b$ by either (i) replacing $x$ by $1$ and $y$ by $0$ or (ii) replacing $x$ by $0$ and $y$ by $1$.

\begin{table}
\centering\begin{tabular}{cccc}
  Elements of $\mathcal P$ & $\rho$ & $\sigma$\\\hline
 $xyxy$   & $\langle i,i+1,i+2,i+3\rangle$   & $\langle i+2,i+1,k+1\rangle$\small\\
 $xyxxy$  & $\langle i,i+1,i+2,i+4\rangle$   & $\langle i+2,i+1,k+1\rangle$\\
 $yxxyx$  & $\langle i+4,i+3,i+2,i\rangle$   & $\langle i+3,i+4,k+1\rangle$\smallskip\\
 $yxxxxy$ & $\langle i+1,i+2,i+4,i+5\rangle$ & $\langle i+4,i+1,i+3\rangle$\\
 $xxxxxy$ & $\langle i+1,i+2,i+4,i+5\rangle$ & $\langle i+4,i+1,i+3\rangle$\smallskip\\
 $xyxxxy$ & $\langle i,i+1,i+2,i+5\rangle$ & $\langle i+2,i+1,i+4\rangle$\\
 $yxxxyx$ & $\langle i+5,i+4,i+3,i\rangle$ & $\langle i+4,i+5,i+2\rangle$\smallskip\\
 $xyyxxx$ & $\langle i,i+3,i+5,i+2\rangle$ & $\langle i+5,i+1,i+4\rangle$\\
 $xxxyyx$ & $\langle i+5,i+2,i,i+3\rangle$ & $\langle i+1,i+5,i+2\rangle$\smallskip\\
 $xyyxxyy$ & $\langle i,i+3,i+4,i+6\rangle$ & $\langle i+5,i+1,i+3\rangle$\smallskip\\
 $xyyyxxx$ & $\langle i,i+4,i+5,i+2\rangle$ & $\langle i+6,i+1,i+4\rangle$
\end{tabular}\caption{Row map $\rho$ and column map $\sigma$ producing $\MIast 3\trans$ or $\overline{\MIast 3\trans}$, where sums involving $i$ are modulo $k$}\label{tab:patterns}
\end{table}
\begin{lem}\label{lem:patterns} Let $k\geq 3$ and let $a$ be a binary sequence of length $k$. Suppose that some $b\in\mathcal P$ occurs in $a$ at position $i$ and that $\rho$ and $\sigma$ are the maps given by the row of Table~\ref{tab:patterns} whose first-column entry is $b$ and where sums involving $i$ are modulo $k$. Then, $(a\miop\MIast k)_{\rho,\sigma}$ equals $\MIast 3\trans$ or $\overline{\MIast 3\trans}$, depending on whether the occurrence of $b$ in $a$ at position $i$ is with $x$ replaced by $1$ and $y$ by $0$, or with $x$ replaced by $0$ and $y$ by $1$, respectively.\end{lem}
\begin{proof} Suppose, for instance, that $b=yxxyx$ occurs in $a$ at position $i$, where $(x,y)=(1,0)$ or $(x,y)=(0,1)$. Thus, $a\miop\MIast k$ contains the following matrix as a configuration, where the labels indicate the rows and columns producing it and sums involving $i$ are modulo $k$:
\def\VR{\kern-\arraycolsep\strut\vrule &\kern-\arraycolsep}
\def\vr{\kern-\arraycolsep & \kern-\arraycolsep}
\[ \bordermatrix{
                ~   & i+1 & i+2 & i+3 & i+4 & \vr k+1 \cr
                i   & x & y & \boldsymbol{y} & \boldsymbol{y} & \VR \boldsymbol{y}\cr
                i+1 & y & y & x & x &\VR x\cr
                i+2 & x & y & \boldsymbol{y} & \boldsymbol{x} & \VR \boldsymbol{x}\cr
                i+3 & y & y & \boldsymbol{x} & \boldsymbol{x} & \VR \boldsymbol{y}\cr
                i+4 & x & x & \boldsymbol{x} & \boldsymbol{y} & \VR \boldsymbol{x}\cr} \]
Bold entries show that, if $\rho$ and $\sigma$ are the maps given by the third row of Table~\ref{tab:patterns}, then $(a\miop\MIast k)_{\rho,\sigma}$ equals $\MIast 3\trans$ or $\overline{\MIast 3\trans}$, depending on whether $(x,y)$ equals $(1,0)$ or $(0,1)$, respectively. For the remaining sequences $b$ in the set $\mathcal P$, the proof is analogous.\end{proof}

A partial converse of the above result is given by our next lemma.

\begin{lem}\label{lem:circRC} If $k\geq 3$ and $a\in A_k$, then the following assertions are equivalent:
\begin{enumerate}[(i)]
 \item\label{itlem:1} $a\miop\MIast k$ contains no $\MIast 3\trans$ and no $\overline{\MIast 3\trans}$ as configurations;

 \item\label{itlem:2} $b$ does not occur in $a$ for any $b\in\mathcal P$;
 
 \item\label{itlem:3} $a\miop\MIast k\in\ForbRowCol$.
\end{enumerate}
\end{lem}
\begin{proof} \eqref{itlem:1}${}\Rightarrow{}$\eqref{itlem:2} Follows from Lemma~\ref{lem:patterns}.

\eqref{itlem:2}${}\Rightarrow{}$\eqref{itlem:3} Suppose no $b\in\mathcal P$ occurs in $a$. If $a$ is constant (i.e., consists entirely of zeros or entirely of ones), then $a\miop\MIast k$ equals $\MIast k$ or $\overline{\MIast k}$, both of which belong $\ForbRowCol$. Thus, we assume, without loss of generality, that $a$ is not constant.

Let $a=a_1a_2\ldots a_k$. Since $a$ is a bracelet which is not constant, $a_1=0$ and $a_k=1$. If $i,j\in[k]$ and $i\leq j$ we say that the ordered pair $(i,j)$ is a \emph{block} if $a_{i-1}\neq a_i=a_{i+1}=\cdots=a_j\neq a_{j+1}$, where $a_0$ stands for $1$ and $a_{k+1}$ stands for $0$. Roughly speaking, a block is a maximal constant consecutive subsequence of $a$. The \emph{length} of a block $(i,j)$ is $j-i+1$. If the length of a block $(i,j)$ is $\ell$, we say that $(i,j)$ is an \emph{$\ell$-block}. We say that block $(i_2,j_2)$ \emph{follows} block $(i_1,j_1)$ if $i_2=j_1+1\pmod k$. A block is \emph{consecutive} to another block if one of them follows the other. By construction, the number $t$ of different blocks is even and, consequently, if $\ell_1,\ldots,\ell_t$ are their lengths, then $k=\ell_1+\cdots+\ell_t$ is a partition of $k$ into an even number of parts.

We have the following facts.
\begin{enumerate}[{Fact }1:] 
 \item\label{claim:1}\emph{If $k\geq 4$, then no two $1$-blocks are consecutive and, as a consequence, at most half of the blocks are $1$-blocks.} In fact, if there were two consecutive $1$-blocks, then $xyxy$ would occur in $a$. Thus, if more than half of the blocks were $1$-blocks, two of them would be consecutive. 
 
 \item\label{claim:2}\emph{If $k\geq 5$, then no $1$-block is consecutive to a $2$-block.} Otherwise, $xyxxy$ or $yxxyx$ would occur in $a$.
 
 \item\label{claim:3}\emph{If $k\geq 6$, then each block has length at most $3$.} Otherwise, $yxxxxy$ or $xxxxxy$ would occur in $a$.

 \item\label{claim:4}\emph{If $k\geq 6$, then no $1$-block is consecutive to a $3$-block.} Otherwise, $xyxxxy$ or $yxxxyx$ would occur in $a$.

 \item\label{claim:5}\emph{If $k\geq 6$, then no $2$-block is consecutive to a $3$-block.} Otherwise, $xyyxxx$ or $xxxyyx$ would occur in $a$.

 \item\label{claim:6}\emph{If $k\geq 7$, then no two $2$-blocks are consecutive.} Otherwise, $xyyxxyy$ would occur in $a$ (because the occurrence of $xyyxxyx$ would contradict Fact~\ref{claim:2}).
 
 \item\label{claim:7}\emph{If $k\geq 7$, then no two $3$-blocks are consecutive.} Otherwise, $xyyyxxx$ would occur in $a$.
\end{enumerate}

If $k\geq 7$, we reach a contradiction because Fact~\ref{claim:3} implies each block has length $1$, $2$, or $3$, but no two such blocks can be consecutive without violating one of the above facts. Thus, we assume, without loss of generality, that $k\leq 6$. Suppose first that $k=6$. Hence, Facts~\ref{claim:1} to~\ref{claim:4} imply that there are neither $1$-blocks nor blocks of length greater than $3$ and, necessarily, $a$ equals $000111$ because $6=3+3$ is the only partition of $6$ into an even number of parts and such that each summand is either $2$ or $3$. Finally, if $k=4$ or $k=5$, then $a$ equals $0001$, $0011$, $0111$, $00001$, $00011$, $00111$, or $01111$ by virtue of Fact~\ref{claim:1} because the only partitions of $4$ and $5$ into an even number of parts and having at most half of the summands equal to $1$ are $4=2+2$, $4=3+1$, $5=4+1$, and $5=3+2$. In all cases, $a\in\ARowCol$ and, by definition, $a\miop\MIast k\in\ForbRowCol$.

\eqref{itlem:3}${}\Rightarrow{}$\eqref{itlem:1} It follows by inspection of the matrices in the set $\ForbRowCol$.\end{proof}

Below, we prove that some matrix in the set $\ForbRowCol\cup\ForbRowCol\trans$ can be found in any matrix not having the circular-ones property for rows and columns, in linear time. We first prove the following intermediate result.

\begin{algorithm2e}[t!]\DontPrintSemicolon
 \KwIn{A matrix $M$ not having the circular-ones property for rows}
 \KwOut{Maps $\rho_D$ and $\sigma_D$ such that $M_{\rho_D,\sigma_D}\in\ForbRowCol\cup\{\MIast 3\trans,\overline{\MIast 3\trans}\}$}

 Find maps $\rho_F$ and $\sigma_F$ such that $M_{\rho_F,\sigma_F}\in\ForbRow$ and the sequence $c$ of entries in the last column of $M_{\rho_F,\sigma_F}$\;\label{algo2:step1}
 
 Let $k$ and $\ell$ be the number of rows and columns, respectively, of $M_{\rho_F,\sigma_F}$\;\label{algo2:step1b}

 \If{$M_{\rho_F,\sigma_F}=c\miop\MIast k$ and some $b\in\mathcal P$ occurs in $c$ at some position $i\in[k]$}
   {\label{algo2:cond1}
       Let $\rho$ and $\sigma$ as specified in the row of Table~\ref{tab:patterns} with first-column entry $b$\;\label{algo2:step2}
       \Return $\rho_D:=\rho\circ\rho_F$ and $\sigma_D:=\sigma\circ\sigma_F$\;\label{algo2:step3}
   }
 \ElseIf{$M_{\rho_F,\sigma_F}\in\{\MIV,\overline\MIV\}$}{\label{algo2:cond2}
      \Return $\rho_D:=\rho_F$ and $\sigma_D:=\langle 6,2,4\rangle\circ\sigma_F$}\label{algo2:step5}
 \Else{\Return $\rho_D:=\rho_F$ and $\sigma_D:=\sigma_F$}\label{algo2:step6}
 \caption{Finds some matrix in the set $\ForbRowCol\cup\{\MIast 3\trans,\overline{\MIast 3\trans}\}$ contained in $M$ as a configuration}\label{algo:2}
\end{algorithm2e}

\begin{lem}\label{lem:algo-circRC} Given a matrix $M$ not having the circular-ones property for rows, Algorithm~\ref{algo:2} finds a matrix in the set $\ForbRowCol\cup\{\MIast 3\trans,\overline{\MIast 3\trans}\}$ contained in $M$ as a configuration. Moreover, Algorithm~\ref{algo:2} can be implemented to run in $O(\size(M))$ time.\end{lem}
\begin{proof} We first prove the correctness. Let $\rho_F$, $\sigma_F$, $c$, $k$, and $\ell$ be as specified in lines~\ref{algo2:step1} and~\ref{algo2:step1b}. If the condition of line~\ref{algo2:cond1} holds, then the output given in line~\ref{algo2:step3} is correct because $M_{\rho\circ\rho_F,\sigma\circ\sigma_F}=(M_{\rho_F,\sigma_F})_{\rho,\sigma}=(c\miop\MIast{k})_{\rho,\sigma}$, which equals $\MIast 3\trans$ or $\overline{\MIast 3\trans}$ by virtue of Lemma~\ref{lem:patterns}. If the condition of line~\ref{algo2:cond2} holds, then the output given in line~\ref{algo2:step5} is correct because $M_{\rho_F,\langle 6,2,4\rangle\circ\sigma_F}=(M_{\rho_F,\sigma_F})_{\id_{4},\langle 6,2,4\rangle}$, which equals $\overline{\MIast 3\trans}$ or $\MIast 3\trans$, depending on whether $M_{\rho_F,\sigma_F}$ equals $\MIV$ or $\overline{\MIV}$, respectively (as can be verified by inspection). It only remains to prove the correctness when the output is given in line~\ref{algo2:step6}. Since $M_{\rho_F,\sigma_F}\in\ForbRow$, either $M_{\rho_F,\sigma_F}=c\miop\MIast k$ (and thus $c\in A_k$) or $M_{\rho_F,\sigma_F}\in\{\MIV,\overline\MIV,\MVast,\overline\MVast\}$. On the one hand, if $M_{\rho_F,\sigma_F}=c\miop\MIast k$, then the output given in line~\ref{algo2:step6} is correct because Lemma~\ref{lem:circRC} implies $c\miop\MIast k\in\ForbRowCol$ under the assumption that the condition of line~\ref{algo2:cond1} does not hold. On the other hand, if $M_{\rho_F,\sigma_F}\in\{\MIV,\overline\MIV,\MVast,\overline\MVast\}$ holds but the condition of line~\ref{algo2:cond2} does not hold, then $M_{\rho_F,\sigma_F}\in\{\MVast,\overline{\MVast}\}\subseteq\ForbRowCol$. This completes the proof of the correctness.

We now prove the linear time bound. By Corollary~\ref{cor:circR}, line~\ref{algo2:step1} can be performed in $O(\size(M))$ time. As the number of elements of $\mathcal P$ is finite, using a linear-time pattern matching algorithm (e.g., Knuth--Morris--Pratt algorithm~\cite{MR0451916}) we can in $O(k)$ time decide whether some $b\in\mathcal P$ occurs in $c$ at some position $i\in[k]$ and, if affirmative, find such $i$. Thus, whether or not the condition of line~\ref{algo2:cond1} holds can be decided in $O(k)$ time because $M_{\rho_F,\sigma_F}=c\miop\MIast k$ holds if and only if $\ell=k+1$. Similarly, whether or not the condition of line~\ref{algo2:cond2} holds can be decided in $O(1)$ time because it is equivalent to deciding whether $(k,\ell)=(4,6)$ and $c\in\{0011,1100\}$ hold. Clearly, each of the remaining lines can be implemented to run in at most $O(k)$ time. As $k$ is at most the number of rows of $M$, the total running time of Algorithm~\ref{algo:2} is $O(\size(M))$.\end{proof}

\begin{thm}\label{thm:algo-circRC} Given a matrix $M$ not having the circular-ones property for rows and columns, a matrix in the set $\ForbRowCol\cup\ForbRowCol\trans$ contained in $M$ as a configuration can be found in $O(\size(M))$ time.\end{thm}
\begin{proof} If $M$ does not have the circular-ones property for rows, then Algorithm~\ref{algo:2} applied to $M$ gives a row map $\rho_D$ and column map $\sigma_D$ of $M$ such that $M_{\rho_D,\sigma_D}$ belongs to $\ForbRowCol\cup\{\MIast 3\trans,\overline{\MIast 3\trans}\}\subseteq\ForbRowCol\cup\ForbRowCol\trans$. Thus, we assume, without loss of generality, that $M$ does not have the circular-ones property for columns; i.e., $M\trans$ does not have the circular-ones property for rows. Applying Algorithm~\ref{algo:2} to $M\trans$, we obtain a row map $\rho_D$ and column map $\sigma_D$ of $M\trans$ such that $D=(M\trans)_{\rho_D,\sigma_D}$ belongs to $\ForbRowCol\cup\{\MIast 3\trans,\overline{\MIast 3\trans}\}$. Hence, $\sigma_D$ is a row map of $M$, $\rho_D$ as a column map of $M$, and $M_{\sigma_D,\rho_D}=D\trans$ belongs to $\ForbRowCol\trans\cup\{\MIast 3,\overline{\MIast 3}\}\subseteq\ForbRowCol\cup\ForbRowCol\trans$.

The $O(\size(M))$ time bound follows from Theorem~\ref{thm:booth-and-lueker} and Lemma~\ref{lem:algo-circRC}.\end{proof}

We are now ready to prove Theorem~\ref{thm:circRC}.

\begin{proof}[Proof of Theorem~\ref{thm:circRC}] We begin by considering the first assertion of the theorem. The `only if' part is clear since $\ForbRowCol\subseteq\ForbRow$. The `if' part follows from Theorem~\ref{thm:algo-circRC}.

That each of the matrices in the set $\ForbRowCol\cup\ForbRowCol\trans$ is a minimal forbidden submatrix for the circular-ones property for rows and columns can be verified by inspection. Hence, in order to prove that $\ForbRowCol\cup\ForbRowCol\trans$ is a minimal set satisfying the first assertion of the theorem, it is enough to observe that any two different matrices in the set $\ForbRowCol\cup\ForbRowCol\trans$ represent different configurations, which follows easily by inspection since the only pairs of matrices in $\ForbRowCol\cup\ForbRowCol\trans$ having the same number of rows, columns, and ones are $\MIast 3$ and $\overline{\MIast 3}$, and $\MIast 3\trans$ and $\overline{\MIast 3\trans}$.\end{proof}

\section{Forbidden subgraphs for concave-round graphs}\label{sec:concave-round}

In this section, we will give a characterization of concave-round graphs by minimal forbidden induced subgraphs and show that one of these forbidden induced subgraphs can be found in linear time in any graph that is not concave-round. Some small graphs needed in what follows are depicted in Figure~\ref{fig:smallgraphs}. If $G$ is a graph, we will denote by $G^*$ the graph that arises from $G$ by adding a single vertex which is adjacent to no vertex of $G$.

\begin{figure}[t!]
\ffigbox[\textwidth]{%
\ffigbox[\FBwidth]{%
\begin{subfloatrow}
\ffigbox[0.18\textwidth]{\includegraphics{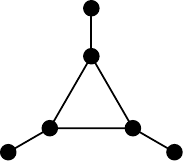}}{\caption{net}}
\ffigbox[0.18\textwidth]{\includegraphics{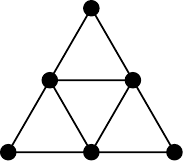}}{\caption{tent}}
\ffigbox[0.18\textwidth]{\includegraphics{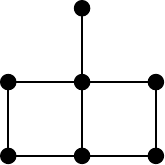}}{\caption{$H_2$}}
\ffigbox[0.18\textwidth]{\includegraphics{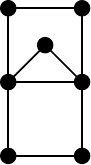}}{\caption{$H_3$}}
\ffigbox[0.18\textwidth]{\includegraphics{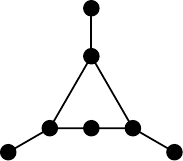}}{\caption{$H_4$}}
\end{subfloatrow}}{}
\bigskip
\ffigbox[\FBwidth]{%
\begin{subfloatrow}
\ffigbox[0.18\textwidth]{\includegraphics{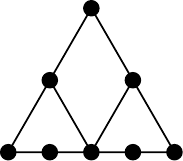}}{\caption{$\BII{2}$}}
\ffigbox[0.18\textwidth]{\includegraphics{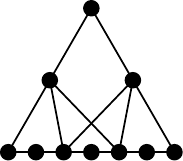}}{\caption{$\BII{3}$}}
\ffigbox[0.18\textwidth]{\includegraphics{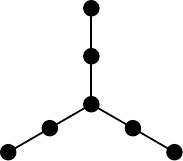}}{\caption{$\BIII{1}$}}
\ffigbox[0.18\textwidth]{\includegraphics{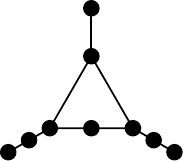}}{\caption{$\BIII{1}$}}
\ffigbox[0.18\textwidth]{\includegraphics{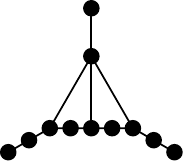}}{\caption{$\BIII{2}$}}
\end{subfloatrow}}{}
}{\caption{Some small graphs}\label{fig:smallgraphs}}
\end{figure}

A bipartite graph $G$ is \emph{biconvex}~\cite{MR1033536} (or \emph{doubly convex}~\cite{MR632418}) if there is a linear ordering of the vertices of $G$ such that the neighborhood of each vertex is an interval in the ordering. We say a graph is \emph{co-biconvex} if its complement is biconvex. Tucker~\cite{MR0379298} proved two characterizations of biconvex graphs, which can be stated in terms of co-biconvex graphs as follows.

\begin{thm}[\cite{MR0379298}]\label{thm:concave-cobip} For each co-bipartite graph $G$, the following assertions are equivalent:
\begin{enumerate}[(i)]
 \item\label{it:cobip1} $G$ is co-biconvex;
 
 \item\label{it:cobip2} $G$ is concave-round;
 
 \item\label{it:cobip3} $G$ contains no induced $\overline{\BII{1}}$, $\overline{\BII{2}}$, $\overline{\BIII{1}}$, $\overline{\BIII{2}}$, $\overline{\BIII{3}}$, or $\overline{C_{2k}}$ for any $k\geq 3$.
\end{enumerate}\end{thm}

The following result reveals an important structural property for concave-round graphs that are not co-bipartite.

\begin{thm}[\cite{MR0309810}]\label{thm:cr-cob->pca} Every concave-round graph which is not co-bipartite is a proper circular-arc graph.\end{thm}

Actually, the connection between proper circular-arc graphs and concave-round graphs is even stronger, as the following result shows.

\begin{thm}[\cite{MR0309810}]\label{thm:pca->cr} Every proper circular-arc graph is concave-round.\end{thm}

Moreover, Tucker also characterized proper circular-arc graphs by minimal forbidden induced subgraphs. 

\begin{thm}[\cite{MR0379298}]\label{thm:pca} A graph is a proper circular-arc graph if and only if it contains none of the following as an induced subgraph: net, tent$^\ast$, $\overline{H_2}$, $\overline{H_3}$, $\overline {H_4}$, $\overline{\BIII{1}}$, $C_k^\ast$ for each $k\geq 4$, $\overline{C_{2k}}$ for each $k\geq 3$, and $\overline{C_{2k+1}^\ast}$ for each $k\geq 1$.\end{thm}

Recently, a linear-time algorithm for finding one of the forbidden induced subgraphs in the above theorem, in any graph that is not a proper circular-arc graph, was devised in~\cite{MR3318808}.

\begin{thm}[\cite{MR3318808}]\label{thm:algo-pca} There is an algorithm that, given any graph $G$  which is not a proper circular-arc graph, finds in linear time a minimal forbidden induced subgraph for the class of proper circular-arc graphs contained in $G$ as an induced subgraph.\end{thm}

The theorem below is our main result and answers the problem posed in~\cite{MR1760336} mentioned in the introduction.

\begin{thm}\label{thm:main} A graph is concave-round if and only if it contains none of the following as an induced subgraph: net, tent$^\ast$, $\overline{H_3}$, $\overline{\BII{1}}$, $\overline{\BII{2}}$, $\overline{\BIII{1}}$, $\overline{\BIII{2}}$, and $\overline{\BIII{3}}$, $C_k^\ast$ for each $k\geq 4$, $\overline{C_{2k}}$ for each $k\geq 3$, and $\overline{C_{2k+1}^\ast}$ for each $k\geq 1$.\end{thm}

We will also prove that one of the minimal forbidden induced subgraphs for the class of concave-round graphs can be found in linear time in any graph that is not concave-round.

\begin{cor}\label{cor:algo-concave-round} There is an algorithm that, given a graph $G$ which is not concave-round, finds in linear time a minimal forbidden induced subgraph for the class of concave-round graphs contained in $G$ as an induced subgraph.\end{cor}

The following lemma is key in our proof of the above two results.

\begin{lem}\label{lem:main} If $H$ is a graph containing an induced subgraph $J$ isomorphic to $H_2$ or $H_4$ and $H$ has a chordless odd cycle $C$, then $H$ contains an induced $\overline{C_4^\ast}$, $C_6$, $H_3$, $\BIII{1}$, or $C_{2k+1}^\ast$ for some $k\geq 1$. 
Moreover, there is an algorithm that, given a graph $G$, an induced subgraph $F$ of $G$ isomorphic to $\overline{H_2}$ or $\overline{H_4}$, and a chordless odd cycle $C$ in $\overline G$, finds an induced subgraph of $G$ isomorphic to $C_4^*$, $\overline{C_6}$, $\overline{H_3}$, $\overline{\BIII{1}}$, or $\overline{C_{2k+1}^*}$ for some $k\geq 1$ in $O(n+m)$ time, where $n$ and $m$ denote the number of vertices and edges of $G$.
\end{lem}
\begin{proof} If a vertex $v$ of $H$ is such that $H-N_H[v]$ has an odd cycle, then $H-N_H[v]$ has some chordless odd cycle $C'$ and, consequently, $V(C')\cup\{v\}$ induces $C_{2k+1}^\ast$ in $H$ for some $k\geq 1$. Hence, in order to prove the first assertion of the lemma, it suffices to prove that either $H-N_H[v]$ has an odd cycle for some $v\in V(H)$ or $H$ contains an induced $\overline{C_4^\ast}$, $C_6$, $H_3$, or $\BIII{1}$ as an induced subgraph.

\begin{figure}
\ffigbox[\textwidth]{%
\begin{subfloatrow}
\ffigbox[0.3\textwidth]{\includegraphics{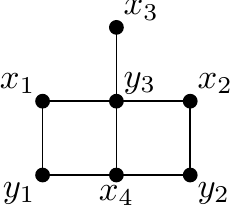}}{\caption{$H_2$}}
\ffigbox[0.3\textwidth]{\includegraphics{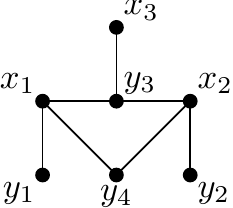}}{\caption{$H_4$}}
\end{subfloatrow}}
{\caption{Labeling of the vertices of $J$ for the proof of Lemma~\ref{lem:main}}\label{fig:labelJ}}
\end{figure}

If some vertex $v$ of $J$ has no neighbor in $V(C)$, then $C$ is an odd cycle in $H-N_H[v]$. Hence, we assume, without loss of generality, that every vertex of $J$ has some neighbor in $V(C)$. We label the vertices of $J$ as in Figure~\ref{fig:labelJ}, depending on whether $J$ is isomorphic to $H_2$ or $H_4$. Let $X=\{x_1,x_2,x_3,x_4\}\cap V(J)$ and $Y=\{y_1,y_2,y_3,y_4\}\cap V(J)$. Let $W_X$ be the set of neighbors of $X$ in $V(C)$, $W_Y$ the set of neighbors of $Y$ in $V(C)$, and $W=W_X\cup W_Y$.

We claim that either $H$ contains $H_3$ or $\BIII{1}$ as an induced subgraph or there is a chordless path $P=a_1,a_2,\ldots,a_s$ in $H$ for some $s\geq 1$ satisfying all the following assertions:
\begin{enumerate}[(i)]
 \item\label{it:P1} $V(P)\subseteq V(C)$;
 \item\label{it:P2} $V(P)\cap W=\{a_1,a_s\}$;
 \item\label{it:P3} $s$ is odd if and only if $\{a_1,a_s\}\cap W_X\neq\emptyset$ and $\{a_1,a_s\}\cap 
 W_Y\neq\emptyset$;
 \item\label{it:P4} if $W_X\cap W_Y\neq\emptyset$, then $s=1$.
\end{enumerate}
If there is some vertex $w\in W_X\cap W_Y$, then the claim holds by letting $s=1$ and $a_1=w$. Hence, we assume in the remaining of this paragraph, without loss of generality, that $W_X\cap W_Y=\emptyset$ (i.e., no vertex of $C$ is adjacent simultaneously to a vertex in $X$ and to a vertex in $Y$). We choose some $w_0\in W_X$ and traverse $C$ (in any of the two possible directions) starting at $w_0$, coloring the traversed vertices with alternating colors $0$ and $1$, starting with color $0$ for $w_0$. We stop whenever a vertex in $W_X$ gets color $1$ or a vertex in $W_Y$ gets color $0$. Since $C$ is odd, the process stops, at the latest, after traversing all of $C$ and recoloring $w_0$ with color $1$. Let $q$ be the last colored vertex (which necessarily belongs to $W$) and let $p$ be the last vertex of $W$ colored before $q$. Let $a_1,a_2,\ldots,a_s$ be the vertices of $C$ traversed from $a_1=p$ to $a_s=q$. By construction, $P=a_1,a_2,\ldots,a_s$ satisfies assertions~\eqref{it:P1} to~\eqref{it:P4} above. In order to complete the proof of the claim, we assume that $P$ is not a chordless path and we will prove that $H$ contains $H_3$ or $\BIII{1}$ as an induced subgraph. As $W_X\cap W_Y=\emptyset$, necessarily $\vert W\vert\geq 2$ and, by construction, $a_1\neq a_s$. As we are assuming that $P$ is not a chordless path in $H$, necessarily $s=\vert V(C)\vert$ and $W=\{a_1,a_s\}$. Since $s$ is odd, we assume, without loss of generality, that $a_1\in W_X$ and $a_s\in W_Y$. As every vertex of $J$ has a neighbor in $W$ and $W_X\cap W_Y=\emptyset$, every vertex of $X$ is adjacent to $a_1$ and every vertex of $Y$ is adjacent to $a_s$. Therefore, $\{a_1,a_2,a_3,x_1,y_1,x_2,y_2\}$ induces $H_3$ or $\BIII{1}$ in $H$ depending on whether $s=3$ or $s\geq 5$, respectively. This completes the proof of the claim.

If $H$ contains $H_3$ or $\BIII{1}$ as an induced subgraph, the first assertion of the lemma holds. Thus, from now on, we assume, without loss of generality, that there is a chordless path $P=a_1,a_2,\ldots,a_s$ in $H$ satisfying assertions \eqref{it:P1} to \eqref{it:P4} of the above paragraph.

We claim that $V(P)\cap V(J)=\emptyset$. Suppose, for a contradiction, that there is some $j\in V(P)\cap V(J)$. Since $V(P)\subseteq V(C)$ and every vertex of $J$ has some neighbor in $X\cup Y$, $j\in W$. Thus, $j\in W\cap V(P)=\{a_1,a_s\}$. If $s=1$, then $j=a_1\in W_X\cap W_Y$, which is a contradiction because no vertex of $J$ has neighbors in $X$ and $Y$ simultaneously. Thus, $s\geq 2$ and, consequently, $a_2\in W$ or $a_{s-1}\in W$ depending on whether $j=a_1$ or $j=a_s$, respectively. Since $W\cap V(P)=\{a_1,a_s\}$, necessarily $s=2$. As $s$ is even, either $a_1,a_2\in W_X$ or $a_1,a_2\in W_Y$. If $a_1,a_2\in W_X$, then $j\in Y$ and the vertex among $a_1$ and $a_2$ which is different from $j$ belongs to $W_X\cap W_Y$, which contradicts $s\neq 1$. Similarly, if $a_1,a_2\in W_Y$, then $j\in X$ and the vertex among $a_1$ and $a_2$ different from $j$ belongs to $W_X\cap W_Y$. These contradictions prove the claim.

The following facts will be useful in the remaining of the proof.
\begin{enumerate}[{Fact }1:]
 \item \emph{If $a_1\in W_X$ and $a_s\in W_Y$, then the only possible neighbor in $P$ of a vertex $x\in X$ is $a_1$.} Let $x\in X$ having some neighbor $w$ in $P$. Since $V(P)\subseteq V(C)$, $w\in W_X\cap V(P)\subseteq W\cap V(P)=\{a_1,a_s\}$. On the one hand, if $w=a_1$, there is nothing to prove. On the other hand, if $w=a_s$, then $a_s\in W_X\cap W_Y$, which implies $s=1$ and, consequently $w=a_1$.

 \item \emph{Symmetrically, if $a_1\in W_X$ and $a_s\in W_Y$, then the only possible neighbor in $P$ of a vertex $y\in Y$ is $a_s$.}
 
 \item \emph{If $a_1\in W_X$, $a_s\in W_X$, and $s\neq 1$, then no vertex in $Y$ has a neighbor in $P$.} Suppose, on the contrary, that some $y\in Y$ has a neighbor $w$ in $P$. Since $V(P)\subseteq V(C)$, $w\in W_Y\cap V(P)\subseteq W\cap V(P)=\{a_1,a_s\}\subseteq W_X$. Hence, $W_X\cap W_Y\neq\emptyset$, which implies $s=1$. This contradiction proves the claim.
 
 \item \emph{Symmetrically, if $a_1\in W_Y$, $a_s\in W_Y$, and $s\neq 1$, then no vertex in $X$ has a neighbor in $P$.}
\end{enumerate}
From this point on, we will repeatedly use the above facts with no further mention to them.
 
Below, we consider all possible cases up to symmetry. In Cases~\ref{case1} to \ref{case9}, $s$ is odd, $a_1\in W_X$, and $a_s\in W_Y$, whereas in Cases~\ref{case10} to \ref{case18}, $s$ is even and $a_1$ and $a_s$ belong both to $W_X$ or both to $W_Y$.
\begin{enumerate}[{Case }1:]
 \item\label{case1}\emph{$s$ is odd, $a_1$ is adjacent to $x_1$, and $a_s$ is adjacent to $y_1$}. If $a_1$ is nonadjacent to $x_2$ or $a_s$ is nonadjacent to $y_2$, then $x_1,a_1,\ldots,a_s,y_1,x_1$ is an odd cycle in $H-N_H[x_2]$ or $H-N_H[y_2]$, respectively. Hence, we assume, without loss of generality, that $x_2$ is adjacent to $a_1$ and $y_2$ is adjacent to $a_s$. If $s=1$, then $\{x_1,y_1,x_2,y_2,a_1\}$ induces $\overline{C_4^*}$ in $H$. Thus, we assume, without loss of generality, that $s\neq 1$. If $a_1$ is nonadjacent to $a_s$, then $\{a_1,x_1,y_1,a_s,y_2,x_2\}$ induces $C_6$ in $H$; otherwise, $\{a_1,x_1,y_1,x_2,y_2,a_2,a_3\}$ induces $H_3$ or $\BIII{1}$ in $H$, depending on whether $s=3$ or $s\geq 5$, respectively.
 
 \item\label{case2}\emph{$s$ is odd, $a_1$ is adjacent to $x_1$, and $a_s$ is adjacent to $y_2$}. If $a_1$ is nonadjacent to $x_3$, then $a_1,a_2,\ldots,a_s,y_2,x_4,y_1,x_1,a_1$ or $a_1,a_2,\ldots,a_s,y_2,x_2,y_4,x_1,a_1$ is an odd cycle in $H-N_H[x_3]$, depending on whether $J$ is isomorphic to $H_2$ or $H_4$, respectively. Hence, we assume, without loss of generality, that $a_1$ is adjacent to $x_3$. If $a_s$ is nonadjacent to $y_1$, then $a_1,a_2,\ldots,a_s,y_2,x_2,y_3,x_3,a_1$ is an odd cycle in $H-N_H[y_1]$; otherwise, we are in Case~\ref{case1}.
 
 \item\label{case3}\emph{$s$ is odd, $a_1$ is adjacent to $x_1$, and $a_s$ is adjacent to $y_3$}. If $a_s$ is nonadjacent to $y_2$, then $a_1,a_2,\ldots,a_s,y_3,x_1,a_1$ is an odd cycle in $H-N_H[y_2]$; otherwise, we are in Case~\ref{case2}.
 
 \item\label{case4}\emph{$s$ is odd, $a_1$ is adjacent to $x_1$, and $a_s$ is adjacent to $y_4$}. If $a_1$ is nonadjacent to $x_3$ or $a_s$ is nonadjacent to $y_2$, then $a_1,a_2,\ldots,a_s,y_4,x_1,a_1$ is an odd cycle in $H-N_H[x_3]$ or $H-N_H[y_2]$, respectively. Hence, we assume, without loss of generality, that $a_1$ is adjacent to $x_3$ and $a_s$ is adjacent to $y_2$. If $a_s$ is nonadjacent to $y_1$, then $a_1,a_2,\ldots,a_s,y_2,x_2,y_3,x_3,a_1$ is an odd cycle in $H-N_H[y_1]$; otherwise, we are in Case~\ref{case1}. 

 \item\label{case5}\emph{$s$ is odd, $a_1$ is adjacent to $x_3$, and $a_s$ is adjacent to $y_1$}. If $a_1$ is adjacent to $x_1$ or $x_2$, then we are in Case~\ref{case1} or in a case symmetric to Case~\ref{case2}, respectively; otherwise, $\{x_1,y_1,x_2,y_2,x_3,y_3,a_1\}$ induces $\BIII{1}$ in $H$.
 
 \item\label{case6}\emph{$s$ is odd, $a_1$ is adjacent to $x_3$, and $a_s$ is adjacent to $y_3$}. If $a_s$ is nonadjacent to $y_1$, then $a_1,a_2,\ldots,a_s,y_3,x_3,a_1$ is an odd cycle in $H-N_H[y_1]$; otherwise, we are Case~\ref{case5}.
 
 \item\label{case7}\emph{$s$ is odd, $a_1$ is adjacent to $x_3$, and $a_s$ is adjacent to $y_4$}. If $a_s$ is nonadjacent to $y_1$, then $a_1,a_2,\ldots,a_s,y_4,x_2,y_3,x_3,a_1$ is an odd cycle in $H-N_H[y_1]$; otherwise, we are in Case~\ref{case5}.

 \item\label{case8}\emph{$s$ is odd, $a_1$ is adjacent to $x_4$, and $a_s$ is adjacent to $y_1$}. If $a_1$ is nonadjacent to $x_2$, then $a_1,a_2,\ldots,a_s,y_1,x_4,a_1$ is an odd cycle in $H-N_H[x_2]$; otherwise, we are in a case symmetric to Case~\ref{case2}.
 
 \item\label{case9}\emph{$s$ is odd, $a_1$ is adjacent to $x_4$, and $a_s$ is adjacent to $y_3$}. If $a_1$ is adjacent to $x_1$ or $x_2$, then we are in Case~\ref{case3} or in a case symmetric to it, respectively. If $a_s$ is adjacent to $y_1$ or $y_2$, we are in Case~\ref{case7} or in a case symmetric to it. Thus, we assume, without loss of generality, that $a_1$ is nonadjacent $x_1$ and $x_2$ and $a_s$ is nonadjacent to $y_1$ and $y_2$. Hence, 
 either $\{x_1,y_1,x_2,y_2,y_3,x_4,a_1\}$ induces $H_3$ in $H$ or $\{x_1,y_1,x_2,y_2,y_4,a_1,a_2\}$ induces $\{x_1,y_1,x_2,y_2,x_4,a_1,a_2\}$ induces $\BIII{1}$ in $H$, depending on whether $s=1$ or $s\geq 3$.
 
 \item\label{case10}\emph{$s$ is even, $a_1$ is adjacent to $x_1$, and $a_s$ is adjacent to $x_3$}. The graph $H-N_H[y_2]$ has the odd cycle $a_1,a_2,\ldots,a_s,x_3,y_3,x_1,a_1$.
 
 \item\label{case11}\emph{$s$ is even, $a_1$ is adjacent to $x_1$, and $a_s$ is adjacent to $x_2$}. If $a_1$ and $a_s$ are nonadjacent to $x_3$, then $a_1,a_2,\ldots,a_s,x_2,y_2,x_4,y_1,x_1,a_1$ or $a_1,a_2,\ldots,a_s,x_2,y_4,x_1$ is an odd cycle in $H-N_H[x_3]$, depending on whether $J$ is isomorphic to $H_2$ or $H_4$, respectively; otherwise, we are in Case~\ref{case10} or in a case symmetric to it.
 
 \item\label{case12}\emph{$s$ is even, $a_1$ is adjacent to $x_1$, and $a_s$ is adjacent to $x_4$}. Without loss of generality, $a_s$ is nonadjacent to $x_2$, since otherwise we are in Case~\ref{case11}. Without loss of generality, $a_1$ is adjacent to $x_2$, since otherwise, $a_1,a_2,\ldots,a_s,x_4,y_1,x_1,a_1$ is an odd cycle in $H-N_H[x_2]$. Without loss of generality, $a_s$ is nonadjacent to $x_1$, since otherwise we are in a case symmetric to Case~\ref{case11}. If $s=2$, then either $\{x_1,x_2,y_1,y_2,x_4,a_1,a_2\}$ induces $H_3$ in $H$ or $\{a_1,x_2,y_2,x_4,y_1,x_1\}$ induces $C_6$ in $H$, depending on whether $a_1$ is adjacent or nonadjacent to $x_4$, respectively; otherwise, $s\geq 4$ and $\{x_1,y_1,x_2,y_2,a_1,a_2,a_3\}$ induces $\BIII{1}$ in $H$.
 
 \item\label{case13}\emph{$s$ is even, $a_1$ is adjacent to $x_3$, and $a_s$ is adjacent to $x_4$}. If $a_1$ is nonadjacent to $x_1$ and $x_2$, then $\{x_1,y_1,x_2,y_2,x_3,y_3,a_1\}$ induces $\BIII{1}$; otherwise, we are in Case~\ref{case12} or a case symmetric to it.
 
 \item\label{case14}\emph{$s$ is even, $a_1$ is adjacent to $y_1$, and $a_s$ is adjacent to $y_2$}. Either 
 $a_1,a_2,\ldots,a_s,y_2,x_4,y_1,a_1$ or $a_1,a_2,\ldots,a_s,y_2,x_2,y_4,x_1,y_1,a_1$ is an odd cycle in $H-N_H[x_3]$, depending on whether $J$ is isomorphic to $H_2$ or $H_4$, respectively.
 
 \item\label{case15}\emph{$s$ is even, $a_1$ is adjacent to $y_1$, and $a_s$ is adjacent to $y_3$}. Without loss of generality, $a_s$ is nonadjacent to $y_2$, since otherwise we are in Case~\ref{case14}. Without loss of generality, $a_1$ is adjacent to $y_2$, since otherwise $a_1,a_2,\ldots,a_s,y_3,x_1,y_1,a_1$ is an odd cycle in $H-N_H[y_2]$. Without loss of generality, $a_s$ is nonadjacent to $y_1$, since otherwise we are in a case symmetric to Case~\ref{case14}. If $s=2$, then either $\{x_1,y_1,x_2,y_2,y_3,a_1,a_2\}$ induces $H_3$ in $H$ or $\{x_1,y_1,x_2,y_2,y_3,a_1\}$ induces $C_6$ in $H$, depending on whether $a_4$ is adjacent or nonadjacent to $y_3$, respectively; otherwise, $s\geq 4$ and $\{x_1,y_1,x_2,y_2,a_1,a_2,a_3\}$ induces $\BIII{1}$ in $H$.
 
 \item\label{case16}\emph{$s$ is even, $a_1$ is adjacent to $y_1$, and $a_s$ is adjacent to $y_4$}. The graph $H-N_H[x_3]$ has the odd cycle $a_1,a_2,\ldots,a_s,y_4,x_1,y_1,a_1$.
 
 \item\label{case17}\emph{$s$ is even, $a_1$ is adjacent to $y_3$, and $a_s$ is adjacent to $y_4$}. If $a_1$ or $a_s$ is adjacent to $y_1$ or $y_2$, then, up to symmetry, we are in Case~\ref{case15} or Case~\ref{case16}; otherwise, $\{x_1,y_1,x_2,y_2,y_3,a_1,a_2\}$ induces $\BIII{1}$ in $H$.
 
 \item\label{case18}\emph{$s$ is even, $a_1$ and $a_s$ have a common neighbor $w$ in $J$}. If $a_1$ or $a_s$ is adjacent to some vertex of $J$ different from $w$, then we are in one of the preceding cases; otherwise, $w,a_1,a_2,\ldots,a_s,w$ is an odd cycle in $H-N_H[v]$, where $v$ is any vertex of $J$ different from $w$ and nonadjacent to $w$.
\end{enumerate}
We have verified that, in all possible cases, either $H-N_H[v]$ has an odd cycle for some $v\in V(H)$ or $H$ contains an induced $\overline{C_4^\ast}$, $C_6$, $\BIII{1}$, or $H_3$. As explained in the first paragraph of the proof, this suffices to prove the first assertion of the lemma.

Consider now the second assertion of the lemma. Suppose we are given a graph $G$, an induced subgraph $F$ of $G$ isomorphic to $\overline{H_2}$ or $\overline{H_4}$, and a chordless odd cycle $C$ in $\overline G$. Let $H=\overline{G[V(F)\cup V(C)]}$ and let $J=\overline F$. By construction, $J$ is an induced subgraph of $H$ isomorphic to $H_2$ or $H_4$ and $C$ is a chordless odd cycle in $H$. Let $n$ and $m$ denote the number of vertices of $G$, respectively. By hypothesis, $\vert V(C)\vert^2\in O(m)$ and, by construction, $\vert V(H)\vert\in O(\vert V(C)\vert)$. Hence, $\vert V(H)\vert^2\in O(m)$. Therefore, $O(n+m)$ time is enough to compute $G[V(F)\cup V(C)]$, compute $H$ as the complement of $G[V(F)\cup V(C)]$, and then apply the constructive proof of the first assertion of the lemma to produce a subset of $V(H)$ inducing in $G$ one of the following graphs: $C_4^*$, $\overline{C_6}$, $\overline{H_3}$, $\overline{\BIII{1}}$, or $\overline{C_{2k+1}}$ for some $k\geq 1$. (In fact, a direct implementation of the constructive proof of the first assertion produces, given $H$, $J$, and $C$, a vertex set inducing $\overline{C_4^*}$, $C_6$, $H_3$, or $\BIII{1}$, or $C_{2k+1}$ for some $k\geq 1$ in $H$ in $O(\vert V(H)\vert+\vert E(H)\vert)$ time.)
\end{proof} 

We are now ready to give the proof of our main result.

\begin{proof}[Proof of Theorem~\ref{thm:main}] For ease of exposition, we introduce $\mathcal F_0=\{\mbox{net},\mbox{tent}^\ast,\overline{H_3},\overline{\BIII{1}}\}\cup\{C_k^\ast\colon\,k\geq 4\}\cup\{\overline{C_{2k}}\colon\,k\geq 3\}\cup\{\overline{C_{2k+1}^\ast}\colon\,k\geq 1\}$, $\mathcal F_1=\{\overline{H_2},\overline{H_4}\}$, and $\mathcal F_2=\{\overline{\BII{1}},\overline{\BII{2}},\overline{\BIII{2}},\linebreak\overline{\BIII{3}}\}$. With this notation, Theorem~\ref{thm:pca} states that proper circular-arc graphs are precisely $(\mathcal F_0\cup\mathcal F_1)$-free graphs and we are to prove that concave-round graphs are precisely $(\mathcal F_0\cup\mathcal F_2)$-free graphs. In addition, let $\mathcal F_{\textrm{co-b}}=\{\overline{\BII{1}},\overline{\BII{2}},\overline{\BIII{1}},\overline{\BIII{2}},\overline{\BIII{3}}\}\cup\{\overline{C_{2k}}\colon\,k\geq 3\}$. With this notation, the equivalence between \eqref{it:cobip2} and \eqref{it:cobip3} in Theorem~\ref{thm:concave-cobip} states that a co-bipartite graph is concave-round if and only if it is $\mathcal F_{\textrm{co-b}}$-free.

As the class of concave-round graphs is hereditary, in order to prove the `only if' part, it suffices to observe that no graph in the set $\mathcal F_0\cup\mathcal F_2$ is concave-round. In fact: (i) since the graphs in $\mathcal F_0$ are not proper circular-arc graphs, those graphs in $\mathcal F_0$ that are not co-bipartite, are not concave-round by virtue of Theorem~\ref{thm:cr-cob->pca}; and (ii) the co-bipartite graphs in $\mathcal F_0$ (namely, $\overline{C_{2k}}$ for each $k\geq 3$ and $\overline{\BIII{1}}$) as well as the graphs in $\mathcal F_2$, all belong to $\mathcal F_{\textrm{co-b}}$ and thus are not concave-round by virtue of Theorem~\ref{thm:concave-cobip}.

In order to prove the `if' part, let $G$ be any minimally not concave-round graph (i.e., $G$ is not concave-round but each induced subgraph of $G$ different from $G$ is concave-round) and suppose, for a contradiction, that $G\notin\mathcal F_0\cup\mathcal F_2$. Because of the minimality of $G$ and the fact that no graph in the set $\mathcal F_0\cup\mathcal F_2$ is concave-round, $G$ is $(\mathcal F_0\cup\mathcal F_2)$-free. Since $G$ is not concave-round, Theorem~\ref{thm:pca->cr} implies that $G$ is not a proper circular-arc graph. Thus, by Theorem~\ref{thm:pca}, $G$ contains some induced subgraph in the set $\mathcal F_0\cup\mathcal F_1$. As $G$ is $\mathcal F_0$-free, $G$ contains $\overline{H_2}$ or $\overline{H_4}$ as an induced subgraph. Notice that $G$ is not co-bipartite, since otherwise the fact that $G$ is a minimal forbidden induced subgraph for the class of concave-round graphs and Theorem~\ref{thm:concave-cobip} would imply that $G\in\mathcal F_{\textrm{co-b}}$, which contradicts $G\notin\mathcal F_0\cup\mathcal F_2$. Hence, there is some chordless odd cycle $C$ in $\overline G$. 
By Lemma~\ref{lem:main}, $G$ contains an induced subgraph in the set $\mathcal F_0$, a contradiction. This contradiction arose from assuming that there was some minimal forbidden induced subgraph $G$ for the class of concave-round graphs such that $G\notin\mathcal F_0\cup\mathcal F_2$. This completes the proof of the `if' part and hence of the theorem.\end{proof}

For the proof of Corollary~\ref{cor:algo-concave-round}, we need a few more results. In the remaining of this section, we assume that the graph $G$ has vertex set $\{1,\ldots,n\}$ and we denote by $M(G)$ the augmented adjacency matrix $M(G)=(m_{ij})$ of $G$ such that $m_{ij}=1$ if and only if $i=j$ or $i$ is adjacent to $j$.

\begin{lem}\label{lem:co--C_2k} Let $G$ be a graph with vertex set $\{1,\ldots,n\}$ and let $\rho$ and $\sigma$ be a row map and a column map, respectively, of the augmented adjacency matrix $M(G)$ of $G$ such that $M(G)_{\rho,\sigma}=\overline{\MI{k'}}$ for some $k'\geq 3$. If the subgraph $G'$ of $G$ induced by $\{\rho(1),\ldots,\rho(k'),\linebreak\sigma(1),\ldots,\sigma(k')\}$ is co-bipartite, then $G'$ is isomorphic to $\overline{C_{2k'}}$.\end{lem}
\begin{proof} Let $x_i=\rho(i)$ and $y_i=\sigma(i)$ for each $i\in[k']$, $X=\{x_1,\ldots,x_{k'}\}$, and $Y=\{y_1,\ldots,y_{k'}\}$. As $\rho$ and $\sigma$ are injective, $\vert X\vert=\vert Y\vert=k'$. Let $G'=G[X\cup Y]$. We assume that $G'$ is co-bipartite and we will prove that $G'$ is isomorphic to $\overline{C_{2k'}}$. Since $M(G)_{\rho,\sigma}=\overline{\MI{k'}}$, the entries $(x_i,y_i)$ and $(x_i,y_{i+1})$ of $M(G)$ are $0$'s for each $i\in[k']$, where $y_{k'+1}$ stands for $y_1$. Thus, $W=y_1,x_1,y_2,x_2,\ldots,y_{k'},x_{k'},y_1$ is a closed walk of length $2k'$ on $\overline{G'}$. As $\overline{G'}$ is bipartite, $X$ and $Y$ are independent sets in $\overline{G'}$ and $X\cap Y=\emptyset$. Therefore, for each $i,j\in[k']$ such that $j\neq i,i+1\pmod{k'}$, the reason why the $(x_i,y_j)$-entry of $M(G)$ is $1$ is that $x_i$ and $y_j$ are different and nonadjacent in $\overline{G'}$. We have proved that $W$ is a chordless cycle on $2k'$ vertices in $\overline{G'}$ and, consequently, $G'$ is isomorphic to $\overline{C_{2k'}}$.
\end{proof}

The following result is a consequence of our findings about the consecutive-ones property for rows and columns in the preceding section.

\begin{algorithm2e}[t!]\DontPrintSemicolon
 \KwIn{A graph $G$ which is not concave-round with vertex set $\{1,\ldots,n\}$}
 \KwOut{Either an induced subgraph of $G$ which is a minimal forbidden induced subgraph for the class of concave-round graphs or a chordless odd cycle in $\overline G$}
 Let $M$ be the augmented adjacency matrix $M(G)$ of $G$\;\label{algo3:step1}
 Find maps $\rho_D$ and $\sigma_D$ such that $M_{\rho_D,\sigma_D}\in\ForbRowCol\cup\{\MIast 3\trans,\overline{\MIast 3\trans}\}$\;\label{algo3:step2}
 Let $k'$ and $\ell'$ be the number of rows and columns of $M_{\rho_D,\sigma_D}$\;\label{algo3:step2b}
 $x_1:=\rho_D(1)$, $x_2:=\rho_D(2)$, \ldots, $x_{k'}:=\rho_D(k')$\;\label{algo3:step3}
 $y_1:=\sigma_D(1)$, $y_2:=\sigma_D(2)$, \ldots, $y_{\ell'-1}:=\sigma_D(\ell'-1)$, $z=\sigma_D(\ell')$\;\label{algo3:step4}
 \If{$k'\leq 6$}{\label{algo3:cond1}
 {
  $G':=G[\{x_1,\ldots,x_{k'},y_1,\ldots,y_{\ell'},z\}]$\;\label{algo3:step5}
  \Return an induced subgraph of $G'$ which is a minimal forbidden induced subgraphs for the class of concave-round graphs\;\label{algo3:step6}}
 }
 \Else
 {\If{$M_{\rho_D,\sigma_D}=\MIast{k'}$}
   {\label{algo3:cond2}
     $G':=G[\{x_1,x_2,x_4,y_1,y_3,y_4\}]\;$\label{algo3:step7}
   }
  \Else{$G':=G[\{x_1,\ldots,x_{k'},y_1,\ldots,y_{k'}\}]$;\label{algo3:step8}}
  \If{$\overline{G'}$ is bipartite}{\label{algo3:cond3}\Return $G'$\;\label{algo3:step9}}
  \Else{\Return a chordless odd cycle in $\overline{G'}$\;\label{algo3:step10}}
 }
 \caption{Finds a minimal forbidden induced subgraph for the class of concave-round graphs or a chordless odd cycle in the complement}\label{algo:3}
\end{algorithm2e}

\begin{thm}\label{thm:forb-or-cycle} Given a graph $G$ which is not concave-round, Algorithm~\ref{algo:3} finds either an induced subgraph of $G$ which is a minimal forbidden induced subgraph for the class of concave-round graphs or a chordless odd cycle in $\overline G$. Moreover, Algorithm~\ref{algo:3} can be implemented to run in linear time.\end{thm}

\begin{proof} We first prove the correctness. Outputs given in lines~\ref{algo3:step6} and \ref{algo3:step10} are correct because, in either case, $G'$ is an induced subgraph of $G$. In only remains to prove that the output given in line~\ref{algo3:step9} is also correct. Hence, we assume, without loss of generality, that the condition of line~\ref{algo3:cond1} does not hold (i.e., $k'\geq 7$ holds) and the condition of line~\ref{algo3:cond3} holds. Suppose first that the condition of line~\ref{algo3:cond2} holds and let $G'$ be defined as in line~\ref{algo3:step7}. In this case, the correctness of the output given in line~\ref{algo3:step9} follows from the fact that if $\rho=\langle x_1,x_2,x_4\rangle$ and $\sigma=\langle y_3,y_4,y_1\rangle$, then $M(G)_{\rho,\sigma}=M(G)_{\langle 1,2,4\rangle\circ\rho_D,\langle 3,4,1\rangle\circ\sigma_D}=\MIast{k'}_{\langle 1,2,4\rangle,\langle 3,4,1\rangle}=\overline{\MI 3}$ and, by virtue of Lemma~\ref{lem:co--C_2k}, $G'$ is isomorphic to $\overline{C_6}$. It only remains to consider the case where the condition of line~\ref{algo3:cond2} does not hold and, consequently, $G'$ is as specified in line~\ref{algo3:step8}. As $M_{\rho_D,\sigma_D}\in\ForbRowCol\cup\{\MIast 3\trans,\overline{\MIast 3\trans}\}$ and none of the conditions of lines~\ref{algo3:cond1} and~\ref{algo3:cond2} holds, necessarily $M_{\rho_D,\sigma_D}=\overline{\MIast{k'}}$. Hence, if $\rho=\langle x_1,\ldots,x_{k'}\rangle$ and $\sigma=\langle y_1,\ldots,y_{k'}\rangle$, then $M(G)_{\rho,\sigma}=\overline{\MI{k'}}$ and, by virtue of Lemma~\ref{lem:co--C_2k}, $G'$ is isomorphic to $\overline{C_{2k'}}$. This completes the proof of the correctness.

As we are working with sparse representations, line~\ref{algo3:step1} can be performed in $O(n+m)$ time. Since $\size(M)\in O(n+m)$, line~\ref{algo3:step2} can completed in $O(n+m)$ time by applying Algorithm~\ref{algo:2}. Lines~\ref{algo3:step2b} to~\ref{algo3:step4} can be performed in $O(n)$ time. If the condition of line~\ref{algo3:cond1} holds, then $\ell'\leq 7$ (because $M_{\rho_D,\sigma_D}\in\ForbRowCol\cup\{\MIast 3\trans,\overline{\MIast 3\trans}\}$) and, consequently, line~\ref{algo3:step6} can be carried out in $O(1)$ time because $G'$, as defined in line~\ref{algo3:step5}, has at most $13$ vertices. Hence, we assume, without loss of generality, that the condition of line~\ref{algo3:cond1} does not hold. Lines~\ref{algo3:cond2} to~\ref{algo3:step8} can be completed in $O(n+m)$ time. It only remains to show that lines~\ref{algo3:cond3} to \ref{algo3:step10} can be performed in $O(n+m)$ time. If the condition of line~\ref{algo3:cond2} holds, then $G'$ has six vertices and lines~\ref{algo3:cond3} to \ref{algo3:step10} can be completed in $O(1)$ time. Therefore, we assume further, without loss of generality, that the condition of line~\ref{algo3:cond2} does not hold. As seen in the preceding paragraph, if $\rho=\langle x_1,\ldots,x_{k'}\rangle$ and $\sigma=\langle y_1,\ldots,y_{k'}\rangle$, then $M(G)_{\rho,\sigma}=\overline{\MI{k'}}$. Thus, if $n'$ and $m'$ are the number of vertices and edges, respectively, of $G'$, then $n'\in O(k')$ and $(k')^2\in O(m')$. Therefore, $(n')^2\in O(m')$, which means that $O(m')$ time suffices to compute $\overline{G'}$, decide whether or not the condition of line~\ref{algo3:cond3} holds and, if not, find the chordless odd cycle required in line~\ref{algo3:step10}. Since $m'\leq m$, lines~\ref{algo3:cond3} and~\ref{algo3:step10} can be performed in $O(n+m)$ time. This completes the proof that Algorithm~\ref{algo:3} can be implemented to run in linear time.\end{proof}

We are now ready to prove Corollary~\ref{cor:algo-concave-round}.

\begin{proof}[Proof of Corollary~\ref{cor:algo-concave-round}] Let $G$ be a graph which is not concave-round. We apply Algorithm~\ref{algo:3} to $G$. If Algorithm~\ref{algo:3} produces an induced subgraph of $G$ which is a minimal forbidden induced subgraph for the class of concave-round graphs, we are done. Hence, we assume, without loss of generality, that Algorithm~\ref{algo:3} produces a chordless odd cycle $C$ in $\overline G$. We now apply the algorithm of Theorem~\ref{thm:algo-pca} to $G$ to produce an induced subgraph $F$ of $G$ which is a minimal forbidden induced subgraph for the class of proper circular-arc graphs. If $F\neq\overline{H_2}$ and $F\neq\overline{H_4}$, then $F$ is a minimal forbidden induced subgraph for the class of concave-round graphs (as seen in the proof of Theorem~\ref{thm:main}) and we are done. Therefore, we assume, without loss of generality, that $F=\overline{H_2}$ or $F=\overline{H_4}$. Applying Lemma~\ref{lem:main} to $G$, $F$, and $C$, an induced subgraph of $G$ which is a minimal forbidden induced subgraph of the class of concave-round graphs can be found in $O(n+m)$ time. This completes the proof of the corollary.\end{proof}

\section{Final remarks}\label{sec:final}

In this section, we point out some connections to other circular-arc graphs. A \emph{normal circular-arc graph}~\cite{MR2368825} is a circular-arc graph admitting a circular-arc model with no two arcs covering the circle (i.e., there are no two arcs whose union is the entire circle). Golumbic~\cite{MR562306} proved that every proper circular-arc graph admits such a model.

\begin{thm}[{\cite[p.~191]{MR562306}}]\label{thm:Golumbic} Every proper circular-arc graph is a normal circular-arc graph.\end{thm}

Recall that concave-round graphs form a superclass of the class of proper circular-arc graphs. We observe, by combining several results in the literature, that, in fact, all concave-round graphs are normal circular-arc graphs. We rely on the results below. For our purposes, it is not necessary to give the definition of interval bigraphs; the interested reader is referred to~\cite{MR1475826}.

\begin{thm}[\cite{MR1475826}]\label{thm:Muller} Every biconvex graph is an interval bigraph.\end{thm}

\begin{thm}[\cite{MR2071482}]\label{thm:HellHuang} The complement of an interval bigraph is a normal circular-arc graph.\end{thm}

By combining the above results with Theorems~\ref{thm:concave-cobip} and~\ref{thm:cr-cob->pca}, the desired conclusion follows.

\begin{cor}\label{cor:Golumbic} Every concave-round graph is a normal circular-arc graph.\end{cor}
\begin{proof} Let $G$ be a concave-round graph. If $G$ is a proper circular-arc graph, then $G$ is a normal circular-arc graph by virtue of Theorem~\ref{thm:Golumbic}. Hence, we assume, without loss of generality, that $G$ is not a proper circular-arc graph. Thus, Theorems~\ref{thm:concave-cobip} and~\ref{thm:cr-cob->pca} imply that $G$ is a co-biconvex graph. Therefore, $\overline G$ is the complement of an interval bigraph by Theorem~\ref{thm:Muller}. Finally, Theorem~\ref{thm:HellHuang} implies that $G$ is a normal circular-arc graph.\end{proof}

Therefore, combining the above result with our Theorem~\ref{thm:main}, we obtain a characterization by minimal forbidden induced subgraphs of those normal circular-arc graphs which are concave-round graphs. A graph is \emph{quasi-line}~\cite{benrebea} if it is $\{\overline{C_{2k+1}^*}:k\geq 1\}$-free.

\begin{cor} A graph is concave-round if and only if it is a quasi-line and $\{\mbox{net},\overline{H_3},\overline{\BII{1}},\overline{\BII{2}},\linebreak\overline{\BIII{1}},\overline{\BIII{2}},\overline{\BIII{3}}\}$-free normal circular-arc graph.\end{cor}
\begin{proof} The `if' part follows from Theorem~\ref{thm:main} since net, $\overline{H_3}$, $\overline{\BII{1}}$, $\overline{\BII{2}}$, $\overline{\BIII{1}}$, $\overline{\BIII{2}}$, $\overline{\BIII{3}}$, and $\overline{C_{2k+1}^*}$ for each $k\geq 1$ are the minimal forbidden induced subgraphs for the class of concave-round graphs which are normal circular-arc graphs (while the remaining forbidden induced subgraphs are not circular-arc graphs). The `only if' part follows from Theorems~\ref{thm:concave-cobip} and~\ref{thm:cr-cob->pca} (as in the second paragraph of the proof of Theorem~\ref{thm:main}) and Corollary~\ref{cor:Golumbic}.\end{proof} 

A circular-arc model is \emph{Helly} if every nonempty subfamily of pairwise intersecting arcs has nonempty total intersection. A \emph{Helly circular-arc graph}~\cite{MR0376439} is a graph having a Helly circular-arc model. A \emph{clique-matrix} $Q(G)$ of a graph $G$ is the incidence matrix of the inclusion-wise maximal cliques versus vertices. The matrix $Q(G)$ is unique up to permutations of rows and of columns. Helly circular-arc graphs are characterized by the circular-ones property for columns of their clique-matrices.

\begin{thm}[\cite{MR0376439}]\label{thm:Gavril} A graph $G$ is a Helly circular-arc graph if and only if $Q(G)$ has the circular-ones property for columns.\end{thm}

Another consequence of our Theorem~\ref{thm:main} is the following.

\begin{cor} A Helly circular-arc graph is concave-round if and only if it is quasi-line.\end{cor}
\begin{proof} It follows from Theorem~\ref{thm:main} because the only minimal forbidden induced subgraphs for the class of concave-round graph which are Helly circular-arc graphs are $\overline{C_{2k+1}^\ast}$ for each $k\geq 1$.\end{proof}

The above corollary will be useful on a subsequent work where, with a different approach, we will give a characterization of the intersection of the classes of concave-round graphs and Helly circular-arc graphs by minimal forbidden induced subgraphs.

Graphs whose clique-matrices have the consecutive-ones property for rows, for columns, or for rows and columns were characterized in the literature as follows. An \emph{interval graph} is the intersection graph of a set of intervals on a line; the set of intervals is called an \emph{interval model} of the graph. These graphs are precisely those whose clique-matrices have the consecutive-ones property for columns.

\begin{thm}[\cite{MR0190028,MR0175811}]\label{thm:FulkersonGross} A graph $G$ is an interval graph if and only if $Q(G)$ has the consecutive-ones property for columns.\end{thm}

A \emph{proper interval graph}~\cite{MR0252267} is the intersection graph of a set of intervals in a line such that no two of them are one a proper subset of the other. Proper interval graphs are also characterized by the consecutive-ones property of their clique-matrices as follows.

\begin{thm}[\cite{MR776781,MR593976,MR0252267}]\label{thm:PIG} For each graph $G$, the following assertions are equivalent:
\begin{enumerate}[(i)]
\item\label{it1:PIG} $G$ is a proper interval graph;
\item\label{it2:PIG} $Q(G)$ has the consecutive-ones property for rows and columns;
\item\label{it3:PIG} $Q(G)$ has the consecutive-ones property for rows. 
\end{enumerate}\end{thm}

Characterizations of those graphs whose clique-matrices have the circular-ones property for columns, or for rows and columns are also known. Recall that, according to Theorem~\ref{thm:Gavril}, Helly circular-arc graphs are those graphs whose clique-matrices have the circular-ones property for columns. A \emph{proper Helly circular-arc graph}~\cite{MR2428582} is a graph admitting a circular-arc model which is simultaneously proper and Helly. The following analogue of the equivalence \eqref{it1:PIG}${}\Leftrightarrow{}$\eqref{it2:PIG} of Theorem~\ref{thm:PIG} was proved in~\cite{MR3030589}.

\begin{thm}[\cite{MR3030589}]\label{thm:PHCAcirc} A graph is a proper Helly circular-arc graph if and only if $Q(G)$ has the circular-ones property for rows and columns.\end{thm}

We conclude this section observing that imposing the circular-ones property just for rows to the clique-matrix also leads to proper Helly circular-arc graphs; i.e., the analogy extends to assertion \eqref{it3:PIG} of Theorem~\ref{thm:PIG}. For that purpose, we rely on the result below. The \emph{claw} is the graph $\overline{C_3^*}$ and the \emph {$k$-wheel} is the graph that arises by adding a single vertex adjacent to every vertex of $C_k$.

\begin{thm}[\cite{MR2428582,MR0379298}]\label{thm:PHCAforb} A graph $G$ is a proper Helly circular-arc graph if and only if $G$ contains no induced claw, $4$-wheel, $5$-wheel, net, tent, $\overline{C_6}$, or $C_k^*$ for any $k\geq 4$.\end{thm}

From the above results, the desired conclusion follows.

\begin{cor} A graph $G$ is a proper Helly circular-arc graph if and only if $Q(G)$ has the circular-ones property for rows.\end{cor}
\begin{proof} By Theorem~\ref{thm:PHCAcirc}, it suffices to prove that if $G$ is not a proper Helly circular-arc graph then $Q(G)$ does not have the circular-ones property for rows. Thus, we assume that $G$ is not a proper Helly circular-arc graph. By Theorem~\ref{thm:PHCAforb}, $G$ contains an induced subgraph $H$ which is isomorphic to claw, $4$-wheel, $5$-wheel, net, tent, $\overline{C_6}$, or $C_k^*$ for some $k\geq 4$. On the one hand, as observed in~\cite{MR3030589}, if $H$ is isomorphic to claw, $4$-wheel, $5$-wheel, or tent, then $Q(H)$ does not have the circular-ones property for rows. On the other hand, if $H$ is isomorphic to net, $\overline{C_6}$, or $C_k^*$ for some $k\geq 4$, then it is also the case that $Q(H)$ does not have the circular-ones property for rows because each of $Q(\mbox{net})$ and $Q(\overline{C_6})$ contains $\MIV$ as a configuration and $Q(C_k^*)$ contains $\MIast k$ as a configuration for each $k\geq 4$. Since $H$ is an induced subgraph of $G$, $Q(H)$ is contained in $Q(G)$ as a configuration. As $Q(H)$ does not have the circular-ones property for rows, $Q(G)$ does not have the circular-ones property for rows.\end{proof}

\section*{Acknowledgements}

This work was partially supported by ANPCyT PICT 2012-1324 and CONICET PIO 14420140100027CO.

\end{document}